\documentclass[10pt,reqno]{article}

\usepackage{amssymb}
\usepackage{epsfig}
\usepackage{amsmath}
\usepackage{amsthm}
\usepackage{color}
\definecolor{r}{rgb}{0.9,0.3,0.1}
\definecolor{b}{rgb}{0.1,0.3,0.9}

\newtheorem{theorem}{Theorem}[section]
\newtheorem{lemma}[theorem]{Lemma}
\newtheorem{corollary}[theorem]{Corollary}

\theoremstyle{remark}
\newtheorem{remark}[theorem]{Remark}

\theoremstyle{definition}
\newtheorem{assumption}[theorem]{Assumption}

\newtheorem{definition}[theorem]{Definition}

\newcommand\cbrk{\text{$]$\kern-.15em$]$}}
\newcommand\opar{\text{\,\raise.2ex\hbox{${\scriptstyle
|}$}\kern-.34em$($}}
\newcommand\cpar{\text{$)$\kern-.34em\raise.2ex\hbox{${\scriptstyle |}$}}\,}

\newcommand{\al}{\alpha}

\newcommand{\ga}{\gamma}

\newcommand\bL{\mathbb{L}}
\newcommand\bR{\mathbb{R}}
\newcommand\bH{\mathbb{H}}
\newcommand\bZ{\mathbb{Z}}

\newcommand\bE{\mathbb{E}}
\newcommand\bM{\mathbb{M}}
\newcommand\bQ{\mathbb{Q}}

\newcommand\cB{\mathcal{B}}

\newcommand\cF{\mathcal{F}}

\newcommand\cO{\mathcal{O}}
\newcommand\cU{\mathcal{U}}
\newcommand\cQ{\mathcal{Q}}

\newcommand\frH{\mathfrak{H}}

\newcommand\aint{-\hspace{-0.38cm}\int}

\newcommand{\mysection}[1]{\section{#1}
\setcounter{equation}{0}}

\begin{document}

\setlength{\baselineskip}{16pt}

\title
{A weighted $L_p$-theory for   parabolic  PDEs with  BMO coefficients on $C^1$-domains}

\author{Kyeong-Hun Kim\footnote{Department of
Mathematics, Korea University, 1 Anam-dong, Sungbuk-gu, Seoul, South
Korea 136-701, \,\, kyeonghun@korea.ac.kr. The research of this
author was supported by Basic Science Research Program through the
National Research Foundation of Korea(NRF) funded by the Ministry of
Education, Science and Technology (2011-0015961)} \qquad
\hbox{\rm and} \qquad Kijung Lee\footnote{Department of Mathematics, Ajou University,
Suwon, South Korea 443-749, \,\, kijung@ajou.ac.kr. The research of this
author was supported by Basic Science Research Program through the
National Research Foundation of Korea(NRF) funded by the Ministry of
Education, Science and Technology (2011-0005597)}}

\date{}


\maketitle

\vspace{0.4cm}

\begin{abstract}
In this paper we   present a weighted $L_p$-theory of
  second-order  parabolic partial differential equations  defined on $C^1$ domains. The leading coefficients are assumed to
be measurable in time variable and   have VMO (vanishing mean oscillation) or small BMO (bounded mean oscillation)  with respect to space variables, and lower order coefficients are allowed to be unbounded and to blow up near the boundary.
   Our   BMO condition is slightly relaxed than the others in the literature.

\vspace*{.125in}


\vspace*{.125in}

\end{abstract}



\mysection{Introduction}

In this article  we are dealing with  a weighted $L_p$-theory of  the parabolic equation:
\begin{eqnarray}
u_t=a^{ij}(t,x)u_{x^ix^j}+b^i(t,x)u_{x^i}+c(t,x)u+f, \quad (t,x)\in (0,T)\times \cO\label{2012.3.22-1}\\
u(t,x)=0,\quad (t,x)\in (0,T)\times \partial \cO \quad ;\quad u(0,x)=u_0(x), \quad x\in \cO,\nonumber
\end{eqnarray}
where  indices $i$ and $j$  run from $1$ to $d$ with   the summation convention on $i$ and $j$ being enforced, and $\cO$ is either a half space or a bounded $C^1$-domain. It is assumed that  leading coefficients $a_{ij}$ are measurable in $t$ and   have  VMO or small  BMO with respect to $x$, and lower order coefficients $b^i$ and $c$ satisfy
\begin{equation}
               \label{eqn on bc}
\lim_{\rho(x)\to 0} \sup_{t} \left( \rho(x)|b^i(t,x)|+\rho^2(x)|c(t,x)|\right)=0,
\end{equation}
where $\rho(x)=\text{dist}(x, \partial \cO)$. Note that (\ref{eqn on bc}) is satisfied if, for instance, $|b^i(t,x)|\leq N\rho^{-1+\varepsilon}(x)$ and $|c(t,x)|\leq N\rho^{-2+\varepsilon}(x)$ for some constants $\varepsilon,N>0$. Also note that     $b^i$ and $c$ are allowed to be unbounded and  blow up near the boundary.

 We look for solutions in function spaces with weights, in which the  derivatives of  solutions are allowed  to blow up near the boundary. In particular, we prove that if $\alpha\in (-1,p-1)$, $u_0=0$ and $\rho f\in L_p((0,T), L_p(\cO, \rho^{\alpha}(x)dx))$,  then equation (\ref{2012.3.22-1})  has a unique solution $u$ so that $u|_{\partial \cO}=0${\color{r}{,}} and for this solution we have
 \begin{equation}
                   \label{eqn 04.09.1}
 \int^T_0\int_{\cO}\left(|\rho^{-1}u|^p+|u_x|^p+|\rho u_{xx}|^p\right) \rho^{\alpha}(x)dxdt\leq N(p,d,c)  \int^T_0\int_{\cO}|\rho f|^p\rho^{\alpha}(x)dxdt.
 \end{equation}
 The condition $\alpha\in (-1,p-1)$ is sharp even for the heat equation $u_t=\Delta u+f$ (see \cite{KL2}). Also{\color{b}{,}} unless much stronger condition on the constant $\alpha$ is imposed, in general (\ref{eqn 04.09.1}) is false  even for the heat equation if $\cO$ is just a Lipschitz domain (see \cite{Kim11}).

 Our motivation of using  such  weighted Sobolev spaces lies in the $L_p$-theory of stochastic partial differential equations (SPDEs) of the type
  \begin{equation}
                      \label{eqn spde}
     dw= (a^{ij}w_{x^ix^j}+b^iw_{x^i}+cw+\tilde{f})dt+(\sigma^{ik}w_{x^i}+g^k)dB^k_t,
  \end{equation}
  where $B^k_t$ ($k=1,2,\cdots)$ are independent one-dimensional Browninan motions defined on a probability space $(\Omega',\cF,P)$, and  all the coefficients and  inputs $\tilde{f}, g^k$ and the solution $w$ are random functions depending also on $(t,x)$.
  It is known that, unless certain compatibility conditions are assumed, the  second derivatives $w_{x^ix^j}$ may blow up near the boundary.   Hence, we have to measure the second derivatives $w_{x^ix^j}$ using appropriate weights near the boundary.
   It is not hard to see that  our weighted $L_p$-theory of equation (\ref{2012.3.22-1}) with BMO coefficients easily yields the corresponding $L_p$-theory for SPDE (\ref{eqn spde}) with BMO coefficients. Indeed, for simplicity assume $b^i=c=\sigma^{ik}=0$ and consider the stochastic heat equation
 \begin{equation}
                 \label{eqn sto heat}
 dv= (\Delta v+\tilde{f})dt+ g^kdB^k_t.
  \end{equation}
  It is well known (e.g. \cite{Kim03,KL2,Lo}) that
  $$
  \bE  \int^T_0\int_{\cO}\left(|\rho^{-1}v|^p+|v_x|^p+|\rho v_{xx}|^p\right) \rho^{\alpha}(x)dxdt\leq N
   \bE \int^T_0\int_{\cO}\left(|\rho \tilde{f}|^p+|g|^p_{\ell_2}+|\rho g_x|^p_{\ell_2} \right) \rho^{\alpha}(x)dxdt,
   $$
   where $\bE X:=\int_{\Omega'}X dP$.
Obviously for each $\omega\in \Omega'$, $\bar{u}:=w-v$ satisfies the deterministic equation
$$
\bar{u}_t=a^{ij}\bar{u}_{x^ix^j}+ (a^{ij}-\delta^{ij})v_{x^ix^j},
$$
and one gets estimates of $\bar{u}$ from (\ref{eqn 04.09.1}) for each $\omega\in \Omega'$. Since $w=v+\bar{u}$,
the weighted $L_p$ norm of $\rho^{-1}
 w,w_x$ and $\rho w_{xx}$ are obtained for free. Therefore inequality (\ref{eqn 04.09.1}) for the deterministic equation yields an extension of existing $L_p$-theories (e.g. \cite{Kim03,Kr99,KL2,Lo})  of SPDE (\ref{eqn spde})  with continuous leading coefficients.

The Sobolev space theory of second-order parabolic and elliptic equations with discontinuous coefficients has been studied extensively in the last few decades. The famous counterexample of Nadirashvili for the solvability of equations with general discontinuous coefficients made people to look for particular type of discontinuity. Among them, VMO condition (or small BMO condition) is very sharp and important from mathematical point of view.
For  practical motivation, we mention that the uniqueness result for elliptic equations  with discontinuous coefficients has connection to the weak uniqueness of solutions of the corresponding stochastic differential equations.

The study of equations with VMO  coefficients was initiated in \cite{CFL} (elliptic equations) and in \cite{BC} (parabolic equations) and  continued in, for instance, \cite{BC}, \cite{By1}, \cite{By2} and \cite{CFL-1}.  In \cite{Kr07} N.V. Krylov gave a unified approach to investigating the $L_p$ solvability of both divergence and non-divergence form of parabolic and elliptic equations with leading coefficients that are measurable in time variable and have VMO (or small BMO) with respect to spatial variables. Since the publication of \cite{Kr07}, the theory kept evolved{\color{r}{,}} especially in the direction of partially VMO coefficients. We refer the reader   to e.g. \cite{DK1}, \cite{DK2} and \cite{DK3}.
The reader can view  our article as a weighted version of existing $L_p$-theories with small BMO (or VMO) coefficients.

Our   BMO (or VMO) condition is slightly relaxed than the others in the literature
(see Remarks \ref{main remark} and \ref{main remark 2}) because we impose small BMO condition only on the balls away from the boundary, that is
 balls of the type $B_r(x)\subset \cO$ with $r\leq \kappa_0\rho(x) \wedge \delta$, where $\delta,\kappa_0\in (0,1)$ are some constants. Thus no restriction is imposed on the balls intersecting with the boundary.
 This relaxation has become possible due to the method  found in \cite{KL11}. The key is to establish weighted sharp function estimate (see Lemma \ref{lemma 5.4.01} below) and apply the weighted version of Fefferman-Stein and Hardy-Littlewood theorems developed in \cite{KL11}.
  By the way, if $a^{ij}$ are continuous in $x$, then our results were already introduced in \cite{KK2, kr99}. Our  article is a natural extension of \cite{KK2,kr99} to the  equations with discontinuous coefficients.

The article is organized as follows. In section \ref{main result} we introduce our weighted Sobolev spaces and the weighted version of Fefferman-Stein and Hardy-Littlewood theorems. In section \ref{section local estimates} we discuss local estimates which we use later. In Section \ref{section sharp} and \ref{section sharp 2} we present sharp function estimates and a priori estimates. In Section \ref{section half spaces} and \ref{section domains} we prove our main results using all previous preparations.

We finish the introduction with some notations. As usual $\bR^{d}$
stands for the Euclidean space of points $x=(x^{1},...,x^{d})$,
$\bR^d_+:=\{x=(x^1,\cdots,x^d)\in \bR^d: x^1>0\}$ and
$B_r(x):=\{y\in \bR^d: |x-y|<r\}$.
 For $i=1,...,d$, multi-indices $\alpha=(\alpha_{1},...,\alpha_{d})$,
$\alpha_{i}\in\{0,1,2,...\}$, and functions $u(x)$ we set
$$
u_{x^{i}}=\frac{\partial u}{\partial x^{i}}=D_{i}u,\quad
D^{\alpha}u=D_{1}^{\alpha_{1}}\cdot...\cdot D^{\alpha_{d}}_{d}u,
\quad|\alpha|=\alpha_{1}+...+\alpha_{d}.
$$
We also use the notation $D^m$ for a partial derivative of order $m$
with respect to $x$; for instance, we use $Du=u_x$ for a first order derivative of $u$ and $D^2u=u_{xx}$ for a second order derivative of $u$. If we write $N=N(a,b,\cdots)$, this means that the
constant $N$ depends only on $a,b,\cdots$. $A\sim B$ means $A\le N_1 B$ and $B\le N_2 A$ for some constants $N_1,N_2$.

The authors are sincerely grateful to Ildoo Kim for finding few errors in the earlier version of this article.

\mysection{Preliminaries: weighted Sobolev spaces on $\bR^d_+$}

                                                    \label{main result}


For any $p>1$ and $\gamma\in \bR$,  define the
space of Bessel potential
$H^{\gamma}_p=H^{\gamma}_p(\mathbb{R}^d)$ as the space
of all distributions $u$ on $\bR^d$ such that
$$
\|u\|_{H^{\gamma}_p}=\|(1-\Delta)^{\ga/2}u\|_{L_p}:=\|\cF^{-1}[(1+|\xi|^2)^{\gamma/2}\cF(u)(\xi)]\|_{L_p}<\infty,
$$
where $\cF$ is the Fourier transform.
Then $H^{\gamma}_p$ is a Banach space with the given norm and
$C^{\infty}_0(\mathbb{R}^d)$ is dense in $H^{\gamma}_p$ (see \cite{T}). If $\gamma$ is a nonnegative integer, then $H^{\ga}_p$ is
the usual Sobolev space,  that is,
$$
H^{\gamma}_p=\{u: D^{\alpha}u \in L_p, |\alpha|\leq \gamma \}, \quad  \quad \|u\|^p_{H^{\gamma}_p}\sim \sum_{|\alpha|\leq \gamma}\int_{\bR^d}|D^{\alpha}u|^p dx.
$$
 It is well known that, for any multi-index $\alpha$, the operator
 $D^{\alpha}:H^{\gamma}_{p}\to
H^{\gamma-|\alpha|}_p${\color{r}{,}} is bounded. On the other hand, if $\text{supp}\,
u \subset (a,b)\times \bR^{d-1}$, where $-\infty<a<b<\infty$, then (see e.g. Remark 1.13 in \cite{kr99})
\begin{equation}
                                        \label{eqn 5.1.1}
\|u\|_{H^{\gamma}_p}\leq N(d,a,b)\|u_x\|_{H^{\gamma-1}_p}.
\end{equation}
Also recall that  if $|\gamma|\leq n$ for some integer $n$ and $|a|_n:=\sup_{|\alpha|\leq n}\sup_x|D^{\alpha}a|<\infty$ then (see e.g. Lemma 5.2 of \cite{Kr99} for a sharper result)
\begin{equation}
                            \label{eqn 10.10.7}
\|au\|_{H^{\gamma}_p}\leq N(d,\gamma)|a|_n \|u\|_{H^{\gamma}_p}.
\end{equation}

Next we recall  definitions and properties of   the weighted Sobolev spaces $H^{\gamma}_{p,\theta}$ introduced in \cite{kr99}  (also see \cite{kr99-1, Lo, Lo2}). The particular case $H^{\gamma}_{2,d}$, i.e. $\theta=d$ and $p=2$, is introduced  in \cite{LM}.
 For $p>1, \theta\in \bR$ and a nonnegative integer $n$ we define
 $$
 H^{n}_{p,\theta}:=\{u: (x^1)^{|\alpha|}D^{\alpha}u\in L_{p}(\bR^d_+, (x^1)^{\theta-d}dx),\,\, |\alpha|\leq n \},
 $$
that is, $u\in H^n_{p,\theta}$ if and only if
\begin{equation}
                     \label{compare}
\sum_{|\alpha|\le n}
\int_{\bR_+^d}|(x^1)^{|\alpha|} D^{\alpha}u(x)|^p(x^1)^{\theta-d} \,dx <\infty.
\end{equation}
We remark that  the space $H^{n}_{p,\theta}$  is different
from $W^{n,p}(\bR^d_+, x^1,\varepsilon)$ introduced in \cite{Ku},
where
\begin{equation}
                              \label{eqn 10.14.2}
W^{n,p}(\bR^d_+, x^1, \varepsilon):=\{u: D^{\alpha}u \in
L_p(\bR^d_+, (x^1)^{\varepsilon}dx),\;|\alpha|\le n\}.
\end{equation}
For general $\gamma \in \bR$ we define the spaces $H^{\gamma}_{p,\theta}$   as follows. Fix a nonnegative  function $\zeta(x)=\zeta(x^1)\in
C^{\infty}_0(\bR_+)$ such that
\begin{equation}
                                       \label{eqn 5.6.5}
\sum_{n=-\infty}^{\infty}\zeta^p(e^{n}x^1)>c>0, \quad \forall x^1\in \bR_+,
\end{equation}
where $c$ is a constant.  Note that    any nonnegative function $\zeta$
with $\zeta>0$ on $[1,e]$ satisfies (\ref{eqn 5.6.5}). For
$\theta\in \bR$, $p>1$ and $\gamma\in \bR$, let
$H^{\gamma}_{p,\theta}=H^{\gamma}_{p,\theta}(\bR^d_+)$
denote the set of all  distributions $u$  on $\bR^d_+$
such that
\begin{equation}
                  \label{def weight}
\|u\|^p_{H^{\gamma}_{p,\theta}}:= \sum_{n\in\bZ} e^{n\theta}
\|\zeta(\cdot) u(e^{n} \cdot)\|^p_{H^{\gamma}_p}<\infty.
\end{equation}
It is not hard to show that   for different $\eta$ satisfying (\ref{eqn 5.6.5}),
we get the same spaces $H^{\gamma}_{p,\theta}$ with equivalent
norms. Indeed, let $\eta(x)=\eta(x^1)\in C^{\infty}_0(\bR_+)$, then there exists an integer $m$ so
that $\xi(x):=\eta(x)[\sum_{n=-\infty}^{\infty} \zeta(e^nx)]^{-1}=\eta(x) [\sum_{|n|\leq m} \zeta(e^nx)]^{-1}\in C^{\infty}_0(\bR_+)$. Thus by (\ref{eqn 10.10.7}),
$$
\|u(e^n\cdot)\eta(\cdot)\|^p_{H^{\gamma}_p}=\|u(e^n\cdot)\xi \sum_{|k|\leq m}\zeta(e^k \cdot)\|^p_{H^{\gamma}_p}
\leq N\sum_{|k|\leq m}\|u(e^n\cdot)\zeta(e^k\cdot)\|^p_{H^{\gamma}_p}\leq N\sum_{|k|\leq m}\|u(e^{n-k}\cdot)\zeta(\cdot)\|^p_{H^{\gamma}_p},
$$
and therefore  we get
\begin{equation}
                            \label{eqn 5.6.1}
\sum_{n=-\infty}^{\infty}
e^{n\theta}\|\eta(\cdot)u(e^n\cdot)\|^p_{H^{\ga}_p} \leq N
\sum_{n=-\infty}^{\infty}
e^{n\theta}\|\zeta(\cdot)u(e^n\cdot)\|^p_{H^{\ga}_p}.
\end{equation}
By the same reason the  reverse of (\ref{eqn 5.6.1}) holds if $\eta$ satisfies  (\ref{eqn 5.6.5}).

  To compare  (\ref{compare}) and (\ref{def weight}) when $\gamma=n=0$,  denote $L_{p,\theta}:=H^0_{p,\theta}$ and note that
$$
\sum_{n}e^{n\theta}\|\zeta(x^1)u(e^nx)\|^p_{L_p}=\int_{\bR^d_+} |u(x)|^p\sum_n e^{n(\theta-d)}\zeta^p(e^{-n}x^1) dx=:\int_{\bR^d_+}|u(x)|^p \eta_0(x^1) dx,
$$
where $\eta_0(x^1):=\sum_n e^{n(\theta-d)}\zeta^p(e^{-n}x^1)$. Obviously
 the function $\xi_0(t):=\sum_n e^{(n-t)(\theta-d)}\zeta^p(e^{t-n})$ is  bounded $1$-periodic function having positive minimum and $\eta_0(x^1)=\xi_0(\ln x^1)(x^1)^{\theta-d}$. It follows that for some $N=N(\zeta)>0$ we have
$$
N^{-1}\|u\|^p_{L_{p,\theta}}\leq \int_{\bR^d_+}|u|^p (x^1)^{\theta-d}dx \leq N \|u\|^p_{L_{p,\theta}}.
$$
Therefore (\ref{compare}) and (\ref{def weight}) give  equivalent norms  if  $\gamma=n=0$.
Actually, in general
 if $\gamma=n$ is a nonnegative integer, then (see Corollary 2.3 of \cite{kr99} for details)
  \begin{eqnarray}
                                 \label{eqn 10.10.1}
 \|u\|^p_{H^{n}_{p,\theta}}\sim \sum_{|\alpha|\le n}
\int_{\bR_+^d}|(x^1)^{|\alpha|} D^{\alpha}u(x)|^p(x^1)^{\theta-d} \,dx.
\end{eqnarray}


 Let $M^{\alpha}$ be the operator of
multiplying by $(x^1)^{\alpha}$ and  $M:=M^1$.  We write $u\in M^{\alpha}H^{\gamma}_{p,\theta}$ if $M^{-\alpha}u \in H^{\gamma}_{p,\theta}$.  For $\nu\in (0,1]$, denote
$$
|u|_{C}=\sup_{x\in\bR^d_+}|u(x)|, \quad [u]_{C^{\nu}}=\sup_{x\neq
y}\frac{|u(x)-u(y)|}{|x-y|^{\nu}}.
$$
Below are other important properties of the spaces
$H^{\gamma}_{p,\theta}$ taken from \cite{kr99, kr99-1}.

\begin{lemma}
                  \label{lemma 1}
Let $\gamma, \theta\in \bR$ and $p\in (1,\infty)$.
\begin{itemize}
\item[(i)] $C^{\infty}_0(\mathbb{R}^d_+)$ is dense in $H^{\gamma}_{p,\theta}$.

\item[(ii)] Assume that $\gamma=m+\nu+d/p$ for some $m=0,1,\cdots$ and
$\nu\in (0,1]$.  Then for any $u\in H^{\gamma}_{p,\theta}$ and $i\in
\{0,1,\cdots,m\}$, we have
\begin{equation}
                         \label{eqn 3.31.4}
|M^{i+\theta/p}D^iu|_{C}+[M^{m+\nu+\theta/p}D^m u]_{C^{\nu}}\leq N
\|u\|_{ H^{\gamma}_{p,\theta}}.
\end{equation}


\item[(iii)] Let $|\gamma|\leq n$ and $|a|^{(0)}_n:=\sup_{|\alpha|\leq n}\sup_{x}M^{|\alpha|}|D^{\alpha}a|<\infty$, then
\begin{equation}
                          \label{eqn 10.10.9}
\|au\|_{H^{\gamma}_{p,\theta}}\leq N(d,\gamma,\theta)|a|^{(0)}_n \|u\|_{H^{\gamma}_{p,\theta}}.
\end{equation}

\item[(iv)] Let $\alpha\in \bR$. Then
$M^{\alpha}H^{\gamma}_{p,\theta+\alpha p}=H^{\gamma}_{p,\theta}$ and
$$
\|u\|_{H^{\gamma}_{p,\theta}}\leq N
\|M^{-\alpha}u\|_{H^{\gamma}_{p,\theta+\alpha p}}\leq
N\|u\|_{H^{\gamma}_{p,\theta}}.
$$

\item[(v)]  $MD, DM: H^{\gamma}_{p,\theta}\to H^{\gamma-1}_{p,\theta}$ are
bounded linear operators, and
$$
\|u\|_{H^{\gamma}_{p,\theta}}\leq N\|u\|_{H^{\gamma-1}_{p,\theta}}+N
\|Mu_x\|_{H^{\gamma-1}_{p,\theta}}\leq N
\|u\|_{H^{\gamma}_{p,\theta}},
$$
$$
\|u\|_{H^{\gamma}_{p,\theta}}\leq N\|u\|_{H^{\gamma-1}_{p,\theta}}+N
\|(Mu)_x\|_{H^{\gamma-1}_{p,\theta}}\leq N
\|u\|_{H^{\gamma}_{p,\theta}}.
$$

\item[(vi)]
If $\theta\neq d-1, d-1+p$, then
\begin{equation}
                        \label{eqn 4.25.3}
\|u\|_{H^{\gamma}_{p,\theta}}\leq N
\|Mu_x\|_{H^{\gamma-1}_{p,\theta}}, \quad  \|u\|_{H^{\gamma}_{p,\theta}}\leq N
\|(Mu)_x\|_{H^{\gamma-1}_{p,\theta}}.
\end{equation}

\item[(vii)] For $i=0,1$\; let $\kappa\in [0,1],\; p_i\in (1,\infty),\; \gamma_i,\;\theta_i \in \bR$ and assume the relations
$$
\gamma=\kappa \gamma_1 +(1-\kappa)\gamma_0,\quad \frac1p=\frac{\kappa}{p_1}+ \frac{1-\kappa}{p_0},\quad
\frac{\theta}{p}=\frac{\theta_1\kappa}{p_1}+\frac{\theta_0(1-\kappa)}{p_0}.
$$
Then
$$
\|u\|_{H^{\gamma}_{p,\theta}}\leq N\|u\|^{\kappa}_{H^{\gamma_1}_{p_1,\theta_1}}\|u\|^{1-\kappa}_{H^{\gamma_0}_{p_0,\theta_0}}.
$$

\end{itemize}
\end{lemma}
\begin{remark}
                \label{remark last}
 Let  $\theta \in (d-1, d-1+p)$ and $n$ be a nonnegative integer. By Lemma \ref{lemma 1} $(iv), (vi)$
\begin{equation}
\|M^{-n}v\|_{H^{\gamma}_{p,\theta}}\le N \|D^n v\|_{H^{\gamma-n}_{p,\theta}}\label{2011.03.21.1}
\end{equation}
 for any $v\in C^{\infty}_0(\mathbb{R}^d_+)$. Indeed, since $\theta+mp\ne d-1,d-1+p$ for any integer $m$
\begin{eqnarray}
\|M^{-n}v\|_{H^{\gamma}_{p,\theta}}&\le& N \|M^{-1}v\|_{H^{\gamma}_{p,\theta-(n-1)p}}\le N \|v_x\|_{H^{\gamma-1}_{p,\theta-(n-1)p}}\nonumber\\
&\le& N\|M^{-1}v_x\|_{H^{\gamma-1}_{p,\theta-(n-2)p}}\le N \|D^2v\|_{H^{\gamma-2}_{p,\theta-(n-2)p}}\ldots.\nonumber
\end{eqnarray}
\end{remark}
\vspace{5mm}

Next, we introduce  Fefferman-Stein and Hardy-Littlewood theorems in weighted $L_p$-spaces.
Denote
$$
\Omega:=\bR \times \bR^d_+:=\{(t,x)=(t,x^1,x^2,\ldots,x^d)\;:x^1>0\}.
$$
   Fix $\alpha \in (-1,\infty)$ and define the weighted measures
$$
\nu(dx)=\nu_{\alpha}(dx)=(x^1)^{\alpha}dx, \quad d\mu=\mu_{\alpha}(dtdx):=\nu_{\alpha}(dx)dt.
$$
Let $B'_r(x')$ denote the open ball in $\bR^{d-1}$ of radius $r$ with center $x'$. For  $x=(x^1,x')\in \bR^d_+$ and $t\in \bR$, denote
$$
\cB_r(x)=\cB_r(x^1,x')=(x^1-r,x^1+r)\times B'_r(x'), \quad \cQ_r(t,x):=(t,t+r^2)\times \cB_r(x).
$$
By $\mathbb{Q}$ we mean the collection of all such open sets $\cQ_r(t,x)\subset \Omega$.
For $f\in L_{1,loc}(\Omega,\mu)$ we define
\begin{eqnarray}
f_{\cQ}=\aint_{\cQ} f\;d\mu,\quad \mathbb{M}f(t,x)=\sup_{\cQ}\aint_{\cQ}
f d\mu,\quad (f)^{\sharp}(t,x)=\sup_{\cQ}\aint_{\cQ}
|f-f_{\cQ}|d\mu{\color{b}{,}} \nonumber
\end{eqnarray}
where the supremum is taken for all $\cQ\in \mathbb{Q}$ containing
$(t,x)$.

\begin{theorem}(\cite{KL11})
                          \label{FS}
$($Fefferman-Stein$)$ Let $p\in (1,\infty)$. Then for any $f\in
L_p(\Omega,\mu)$,  we have
\begin{eqnarray}
\|f\|_{L_p(\Omega,\mu)}\le
N\|f^{\sharp}\|_{L_p(\Omega,\mu)},
\nonumber
\end{eqnarray}
where $N=N(\alpha,p,d)$.
\end{theorem}

\begin{theorem}(\cite{KL11})
                       \label{HL}
$($Hardy-Littlewood$)$ Let $p\in (1,\infty)$. Then for $f\in
L_p(\Omega,\mu)$ we have
\begin{eqnarray}
\|\mathbb{M}f\|_{L_p(\Omega,\mu)}\le N
\|f\|_{L_p(\Omega,\mu)},  \nonumber
\end{eqnarray}
where $N=N(\alpha,p,d)$.
\end{theorem}

\mysection{Some local estimates  of solutions}
                                       \label{section local estimates}

In this section we develop some local estimates of $D^{\beta}u$ for any multi-index $\beta$,   where $u$ is a solution of the equation:
\begin{equation}
                                 \label{eqn 10.11.1}
   u_t+a^{ij}u_{x^ix^j}=f, \quad (t,x)\in \Omega:= \bR \times \bR^d_+.
  \end{equation}
  In particular, we prove that if $f=0$ in $\cQ_r(r):=(0,r^2)\times (0,2r)\times B'_r(0)$ then for any $s\in (0,r)$ and  $\theta\in (d-1,d-1+p)$,
  \begin{eqnarray*}
\max_{(t,x)\in \cQ_s(s)}(|D^{\beta}u_{xx}|^p+|D^{\beta}u_{t}|^p)\le N(r,s,\beta,\theta)
\int_{\cQ_r(r)}|u|^p (x^1)^{\theta-d+p}dxdt.
\end{eqnarray*}

  The estimates obtained here will be used  to estimate the sharp function of $u_{xx}$ in the next section.

Throughout this section we assume the following.
\begin{assumption}
                     \label{main assumption2}
$a^{ij}=a^{ij}(t)$ are independent of $x$, and there exist  constants
$\delta, K>0$ so that
\begin{equation}
                    \label{assumption 1}
\delta|\xi|^2\leq a^{ij}(t) \xi^i\xi^j \leq K|\xi|^2, \quad \forall \xi \in \bR^d.
\end{equation}
\end{assumption}

For $-\infty \leq S<T\leq \infty$, we define the Banach spaces
$$
\bH^{\gamma}_{p,\theta}(S,T):=L_p((S,T),H^{\gamma}_{p,\theta}), \; \bH^{\gamma}_{p,\theta}(T):=\bH^{\gamma}_{p,\theta}(0,T),\;
\bL_{p,\theta}(S,T):=H^0_{p,\theta}(S,T),\; \bL_{p,\theta}(T):=\bL_{p,\theta}(0,T)
$$
with the norms given by
\[
\|u\|_{\bH^{\ga}_{p,\theta}(S,T)}=\left[\int^{T}_S\|u(t)\|^p_{H^{\gamma}_{p,\theta}}dt\right]^{1/p}.
\]
Finally, we set
$U^{\gamma}_{p,\theta}:=M^{1-2/p}H^{\gamma-2/p}_{p,\theta}$ with the norm
$$\|u\|_{U^{\gamma}_{p,\theta}}:=\|M^{-1+2/p}u\|_{H^{\gamma-2/p}_{p,\theta}}.
$$

First we recall  a Krylov's  result   for  equations with coefficients independent of $x$.

\begin{lemma}
                      \label{lem constant}
 Let $d-1<\theta<d-1+p$, $p\in (1,\infty), \gamma \in \bR$ and $T\in (0, \infty]$. Then for any $f\in M^{-1}\bH^{\gamma}_{p,\theta}(T)$ and
 $u_0 \in U^{\gamma+2}_{p,\theta}$, the
 equation
 \begin{equation}
                 \label{single eqn}
                 u_t=a^{ij}u_{x^ix^j}+f, \quad \quad u(0)=u_0
                 \end{equation}
  has a unique solution (in the sense of distributions, see Remark \ref{remark 10.1} below) $u$ {\color{b}{in}} $M\bH^{\gamma+2}_{p,\theta}(T)$, and for this solution
  \begin{equation}
                                    \label{eqn estimate 1}
     \|M^{-1}u\|_{\bH^{\gamma+2}_{p,\theta}(T)}\leq N\left(\|Mf\|_{\bH^{\gamma}_{p,\theta}(T)}+\|u_0\|_{U^{\gamma+2}_{p,\theta}}\right),
     \end{equation}
     where $N=N(\delta, K, \theta,\gamma,p)$.
     \end{lemma}

\begin{proof}
See Theorem 5.6 of \cite{kr99}.
\end{proof}

For any distribution $h$ on $\bR^d_+$ and $\phi\in C^{\infty}_0(\bR^d_+)$, by $(h,\phi)$ we denote the image of $\phi$ under $h$.
\begin{remark}
                \label{remark 10.1}
   We say that $u$ is a solution of (\ref{single eqn}) in the sense of distributions if for any $\phi \in C^{\infty}_0(\bR^d_+)$
   $$
   (u(t),\phi)=(u_0,\phi)+\int^t_0 (a^{ij}u_{x^ix^j}+f,\phi)ds, \quad \quad \forall \, t\leq T{\color{b}{.}}{\color{r}{,}}
   $$
   \end{remark}

   \begin{corollary}
                            \label{cor 5.1.1}
   Let $-\infty \leq S<T<\infty$. For any $f\in M^{-1}\bH^{\gamma}_{p,\theta}(S,T)$ and $u_0 \in U^{\gamma+2}_{p,\theta}$, the
 equation
 \begin{equation}
                 \label{single eqn 2}
                 u_t+a^{ij}u_{x^ix^j}=f, \quad t\in (S,T)
                 \end{equation}
 with $u(T)=u_0$ has a unique solution $u$ in $M\bH^{\gamma+2}_{p,\theta}(S,T)$, and for this solution
  \begin{equation}
                                    \label{eqn estimate 2}
     \|M^{-1}u\|_{\bH^{\gamma+2}_{p,\theta}(S,T)}\leq N\left(\|Mf\|_{\bH^{\gamma}_{p,\theta}(S,T)}+\|u_0\|_{U^{\gamma+2}_{p,\theta}}\right),
     \end{equation}
     where $N=N(\delta, K, \theta,\gamma,p)$.
     \end{corollary}

     \begin{proof}
     It is enough to consider the time change $t\to -(t-T)$ and use Lemma \ref{lem constant}.
     \end{proof}

Denote
$$
\cQ_r(a)=(0,r^2)\times (a-r,a+r)\times B'_r(0),\quad U_r=(-r^2,r^2)\times (-2r,2r)\times B'_r(0).
$$

\begin{lemma}
                 \label{lemma 3}
Let $d-1<\theta<d-1+p$, $0<s<r <\infty$, $u(t,x)\in C^{\infty}_0(\bR\times \bR^d_+)$ and
$$
u_t+a^{ij}(t) u_{x^ix^j}=0  \quad \quad \text{for} \quad (t,x)\in \cQ_r(r).
$$
 Then for any multi-index  $\beta=(\beta^1,\cdots,\beta^d)$, we have
\begin{eqnarray}
                                   \label{eqn new}
&& \int_{\cQ_s(s)}\left(|M^{-1}D^{\beta}u|^p+|D^{\beta}u_x|^p+|MD^{\beta}u_{xx}|^p\right)(x^1)^{\theta-d}dxdt \nonumber\\
 &\leq& N(1+r)^{|\beta|p}\cdot (1+(r-s)^{-2})^{(|\beta|+1)p}
\int_{\cQ_r(r)}|Mu(t,x)|^p (x^1)^{\theta-d} dxdt,
\end{eqnarray}
where $N=N(\theta,p,|\beta|, \delta,K)$.
\end{lemma}

\begin{proof}
We use the induction on $|\beta|$.

 First, let $|\beta|=0$. We modify the proof of Lemma 2.4.4 of \cite{kr08}.  Denote $r_0=s$ and     $r_m=s+(r-s)\sum_{j=1}^m2^{-j}$ for $m=1,2,\cdots$.
Choose  smooth functions $\zeta_m\in C^{\infty}_0(\bR^{d+1})$ so that  $0\leq \zeta_m\leq 1$,
$$
\zeta_m=1 \quad \text{on} \quad U_{r_m}, \quad \quad \zeta_m=0 \quad \text{on}\quad \bR^{d+1} \setminus U_{r_{m+1}},
$$
\begin{equation}
                 \label{eqn 10.10.5}
|\zeta_{mx}|\leq N (r-s)^{-1}2^m, \quad |\zeta_{mxx}|\leq N(r-s)^{-2} 2^{2m},\quad |\zeta_{mt}|\leq N (r-s)^{-2} 2^{2m}.
\end{equation}
Note that for each $m$, $(u\zeta_m)(r^2,x)=0$  and  $u\zeta_m$ satisfies
$$
(u\zeta_m)_t+a^{ij}(u\zeta_m)_{x^ix^j}=f_m:=\zeta_{mt}u+a^{ij}u\zeta_{mx^ix^j}+2a^{ij}(u\zeta_{m+1})_{x^i}\zeta_{mx^j}, \quad \quad (t,x)\in (0,r^2)\times \bR^d_+.
$$
By Corollary \ref{cor 5.1.1} for $\gamma=0$,
$$
A_m:=\|M^{-1}u\zeta_m\|_{\bH^2_{p,\theta}(r^2)}\leq N \|Mf_m\|_{\bL_{p,\theta}(r^2)}.
$$
Denote $B:= (\int_{\cQ_r(r)}|Mu|^p (x^1)^{\theta-d}dxdt)^{1/p}$. Then by (\ref{eqn 10.10.5}) and  Lemma \ref{lemma 1},
\begin{eqnarray*}
&&\|\zeta_{mt}Mu+a^{ij}Mu\zeta_{mx^ix^j}\|_{\bL_{p,\theta}(r^2)} \leq N(r-s)^{-2}2^{2m}B,\\
&&\|a\zeta_{mx}M(u\zeta_{m+1})_{x}\|_{\bL_{p,\theta}(r^2)}\leq N(r-s)^{-1}2^{m}\|M(u\zeta_{m+1})_x\|_{\bL_{p,\theta}(r^2)}
\leq N(r-s)^{-1}2^{m}\|u\zeta_{m+1}\|_{\bH^1_{p,\theta}(r^2)}.
\end{eqnarray*}
By Lemma \ref{lemma 1} (vii) (take $p_0=p_1=p, \gamma=1,\gamma_0=0,\gamma_1=2,\theta_0=\theta+p,\theta_1=\theta-p$ and $\kappa=1/2$)  for any $\varepsilon>0$
$$
(r-s)^{-1}2^{m}\|u\zeta_{m+1}\|_{\bH^1_{p,\theta}(r^2)}\leq \varepsilon A_{m+1}+ \varepsilon^{-1}(r-s)^{-2}2^{2m}B.
$$
It follows that  (with $\varepsilon$ different from the one above),
$$
A_{m}\leq \varepsilon A_{m+1}+ N(1+\varepsilon^{-1})(r-s)^{-2}2^{2m}B.
$$
We take $\varepsilon=\frac{1}{16}$ and get
$$
\varepsilon^m A_m\leq  \varepsilon^{m+1}A_{m+1}+N\varepsilon^m(1+\varepsilon^{-1})2^{2m}(r-s)^{-2}B,
$$
$$
A_0+\sum_{m=1}^{\infty}\varepsilon^mA_m\leq \sum_{m=1}^{\infty}\varepsilon^mA_m+N(r-s)^{-2}B.
$$
Note that the series $\sum_{m=1}\varepsilon^mA_m$ converges because  $A_m\leq N2^{2m}\|M^{-1}u\|_{\bH^2_{p,d}(r^2)}$.   By Lemma \ref{lemma 1}(v) and (vi), for any $M^{-1}w\in H^2_{p,\theta}$,
\begin{equation}
                              \label{eqn 4.24.6}
      \|M^{-1}w\|_{H^2_{p,\theta}}     \sim (\|M^{-1}w\|_{L_{p,\theta}}+\|w_x\|_{L_{p,\theta}}+\|Mw_{xx}\|_{L_{p,\theta}}).
 \end{equation}
 Therefore,
$$
\int_{\cQ_s(s)}\left(|M^{-1}u|^p+|u_x|^p+|Mu_{xx}|^p\right) (x^1)^{\theta-d}dxdt\leq N A^p_0 \leq N(r-s)^{-2p}B^p,
$$
and (\ref{eqn new}) is proved for $|\beta|=0$.

 Next, assume that (\ref{eqn new}) holds  whenever $s<r$ and  $|\beta'|=k$, that is
\begin{eqnarray}
&& \int_{\cQ_s(s)}\left(|M^{-1}D^{\beta'}u|^p+|D^{\beta'}u_x|^p+|MD^{\beta'}u_{xx}|^p\right)(x^1)^{\theta-d}dxdt \nonumber\\
 &\leq& N(1+r)^{kp}\cdot (1+ (r-s)^{-2})^{(k+1)p}
\int_{\cQ_r(r)}|Mu(t,x)|^p (x^1)^{\theta-d} dxdt. \label{eqn 10.1.1}
\end{eqnarray}
Let $|\beta|=k+1$ and $D^{\beta}=D_iD^{\beta'}$ for some $i$ and $\beta'$ with $|\beta'|=k$.  Fix
  a smooth function $\eta$  so that  $\eta=1$ on $U_s$, $\eta=0$ on $\bR^{d+1}\setminus U_{(r+s)/2}$, $|\eta_x|\leq N(r-s)^{-1}, |\eta_{xx}|\leq N(r-s)^{-2}$ and
  $|\eta_t|\leq N(r-s)^{-2}$.
 Note that $v:=\eta D^{\beta}u$ satisfies $v(r^2,\cdot)=0$ and
$$
v_t+a^{ij}v_{x^ix^j}=f:=\eta_{t}D^{\beta}u+2a^{ij}\eta_{x^i}D^{\beta}u_{x^j}+a^{ij}\eta_{x^ix^j}D^{\beta}u, \quad \quad (t,x)\in (0,r^2)\times \bR^d_+.
$$
By Corollary \ref{cor 5.1.1} for $\gamma=0$ (also note that $x^1\leq r$ on the support of $\eta$ and $(r-s)^{-1}\leq 1+(r-s)^{-2}$),
\begin{eqnarray*}
\|M^{-1}v\|^p_{\bH^2_{p,\theta}(r^2)} &\leq& N \|M\eta_{t}D^{\beta}u+2a\eta_{x}MD^{\beta}u_x+Ma\eta_{xx}D^{\beta}u\|^p_{\bL_{p,\theta}(r^2)}\\
&\leq&N (1+r)^p(1+(r-s)^{-2})^p \int_{\cQ_{(s+r)/2}((s+r)/2)}\left(|D^{\beta}u|^p+|MD^{\beta}u_x|^p\right)\mu(dtdx)\\
&\leq&N (1+r)^p(1+(r-s)^{-2})^p \int_{\cQ_{(s+r)/2}((s+r)/2)}\left(|D^{\beta'}u_x|^p+|MD^{\beta'}u_{xx}|^p\right)\mu(dtdx).
\end{eqnarray*}
This, (\ref{eqn 4.24.6}) and (\ref{eqn 10.1.1}) show that  the induction goes through, and hence the lemma is proved.
\end{proof}

The following result can be found e.g. in \cite{KL11}, and we give a outline of the proof for the sake of the completeness.
\begin{lemma}
                   \label{lemma 1111}
Let $u(t,x)\in C^{\infty}_0(\bR\times \bR^d_+)$. Then for any
$T>0$, $p>1$ and $n=0,1,2,\cdots$,
$$
\sup_{t\in [0,T]}\|u(t,\cdot)\|_{H^n_{p,\theta}}\leq
N(\|u\|_{\bH^n_{p,\theta}(T)}+\|u_t\|_{\bH^n_{p,\theta}(T)}).
$$
\end{lemma}
\begin{proof}
First of all, it is easy to check that for any $\phi=\phi(t)\in W^1_p((0,T))$ (see \cite{kr08}, p.32)
$$
\sup_{t}|\phi(t)|^p\leq N \int^t_0(|\phi|^p+|\phi'(t)|^p)dt.
$$
Thus  it suffices to prove
\begin{equation}
                                          \label{eqn 12.19}
\phi(t):=\|u(t,\cdot)\|_{H^n_{p,\theta}}\in W^1_p((0,T)), \quad |\phi'(t)|\leq\|u_t(t,\cdot)\|_{H^n_{p,\theta}}.
\end{equation}
One can  prove (\ref{eqn 12.19}) by repeating  the proof of Exercise 2.4.8 of \cite{kr08} (see p.71). It is enough to replace $H^n_p$ there by $H^n_{p,\theta}$.
\end{proof}

By $C^{\infty}_{loc}(\Omega)$ we denote the set of
real-valued functions $u$ defined on $\Omega$ and such
that $\zeta u\in C^{\infty}_0(\Omega)$ for any
$\zeta\in C^{\infty}_0(\Omega)$.

\begin{lemma}\label{101004.14.53}
Let  $\theta\in (d-1,d-1+p)$,  $s\in (0,r)$ and  $u\in C^{\infty}_{loc}(\Omega)$ satisfies\;
$u_t+a^{ij}(t)u_{x^ix^j}=0$ for $(t,x)\in \cQ_r(r)$. Then for any multi-index $\beta=(\beta^1,\beta^2,\cdots,\beta^d)$,
\begin{eqnarray}
\max_{(t,x)\in \cQ_s(s)}(|D^{\beta}u_{xx}|^p+|D^{\beta}u_{t}|^p)\le N
\int_{\cQ_r(r)}|u|^p (x^1)^{\theta-d+p}dxdt,\nonumber
\end{eqnarray}
where  $N=N(\theta, s,r,\beta,p,\delta,K)$.
\end{lemma}

\begin{proof}  Choose the smallest integer $n$ so that $np>(\theta \vee d)$.
Note that if  $v\in C^{\infty}_0(\bR^d_+)$ and  $v(x)=0$ for $x^1 \geq r$, then by Lemma \ref{lemma 1}(ii) with $\gamma=n$, $i=0$  and $u=M^{-n}v$,
\begin{equation}
              \label{eqn 12.10.1}
\sup_{x}|v(x)|\leq N(r)\sup_{x}|M^{\theta/p}M^{-n}v(x)|\leq N\|M^{-n}v\|_{H^n_{p,\theta}}\leq N(r,p,n)\|D^nv\|_{L_{p,\theta}},
\end{equation}
where for the last inequality we use Remark \ref{remark last}.

Fix $\kappa\in (s,r)$. Let  $\psi$  be a smooth function so that
$\psi(t,x)=1$ for $(t,x)\in \cQ_{s}(s)$ and $ \psi=0$ for $(t,x)\not\in U_{\kappa}$.  Then $\psi_x=\psi_t=0$ on $\cQ_s(s)$. It follows from (\ref{eqn 12.10.1}), Lemma \ref{lemma 1111} and Lemma \ref{lemma 3}  that
\begin{eqnarray*}
\max_{\cQ_{s}(s)} \left(|D^{\beta}u_{xx}|^p+|D^{\beta}u_{t}|^p\right)&\leq& N\max_{(t,x)\in \cQ_{s}(s)} | (D^{\beta}\psi u)_{xx}|^p\\
&\leq &N\max_{t\in [0,s^2]}\|D^n_x
(D^{\beta}\psi u)_{xx}\|^p_{L_{p,\theta}}
\\
 &\leq& N
\left(\|D^n_x(D^{\beta}\psi u)_{xx}\|^p_{\bL_{p,\theta}(s^2)}+\| D^n_x(D^{\beta}\psi u_t)_{xx}\|^p_{\bL_{p,\theta}(s^2)}\right)\\
&\leq &N \sum_{|\alpha|\leq n+|\beta|+4}\int_{Q_{\kappa}(\kappa)}|D^{\alpha}u|^p (x^1)^{\theta-d}\,dxdt\\
& \leq& N  \int_{\cQ_r(r)}|Mu|^p(x^1)^{\theta-d}dxdt.
\end{eqnarray*}
The lemma is proved.
 \end{proof}

\mysection{Sharp function estimates for equations with  coefficients independent of $x$}
       \label{section sharp}

In this section we introduce some  results  developed in \cite{KL11} with detailed proofs for the sake of completeness, and extend Theorem \ref{101004.16.53} and Theorem \ref{theorem 5.2.1} to wider range of weights. These theorems are proved in \cite{KL11} only for $\theta\in (d-1,d]$ and we extend
them  for any $\theta\in (d-1,d-1+p)$.

Denote $\nu^1_{\alpha}(dx^1)=(x^1)^{\alpha}dx^1$. Recall that
$$\nu_{\al}(dx)=(x^1)^{\alpha}dx^1dx'=\nu^1_{\al}(dx^1)dx', \quad \quad \cB_r(a)=(a-r,a+r)\times B'_r(0).
$$

 We start with a weighted Poincar\'e's inequality.

\begin{lemma}(\cite{KL11})
                                  \label{2010.03.23.3}
Let $\alpha>0$, $p\in [1,\infty)$, $\cB_r(a)\subset \mathbb{R}^d_+$,
and $u\in C^{\infty}_{loc}(\mathbb{R}^d_+)$. Then
\begin{eqnarray}
\int_{\cB_r(a)}\int_{\cB_r(a)}|u(x)-u(y)|^p \nu_{\al}(dx)\; \nu_{\al}(dy)\le
2^{\alpha+1} (2r)^p\nu_{\al}(\cB_r(a))\int_{\cB_r(a)}|u_x(x)|^p
\nu_{\al}(dx).\label{2010.03.17.1}
\end{eqnarray}
\end{lemma}
\begin{proof}
For $x,y\in \cB_r(a)$  we have
\begin{eqnarray}
|u(x)-u(y)|^p\le (2r)^p\int^1_0 |u_x(tx+(1-t)y)|^pdt\nonumber
\end{eqnarray}
and the left-hand side of $(\ref{2010.03.17.1})$ is less than or equal to
\begin{eqnarray}
(2r)^p\int^1_0 I(t)dt=2(2r)^p \int^1_{1/2} I(t)dt,\nonumber
\end{eqnarray}
where
\begin{eqnarray}
I(t):=\int_{\cB_r(a)}\int_{\cB_r(a)}|u_x(tx+(1-t)y)|^p \nu_{\al}(dx)\;
\nu_{\al}(dy)\nonumber
\end{eqnarray}
and $I$ satisfies $I(t)=I(1-t)$. For each $t\in [1/2,1]$ and $y$, $t\cB_r(a)+(1-t)y:=\{tz+(1-t)y:z\in
\cB_r(a)\}\subset \cB_r(a)$.
Substituting $w=tx+(1-t)y$ and noticing $x^1=(w^1-(1-t)y^1)/t\leq w^1/t$ since $y^1\geq 0$, we get
\begin{eqnarray}
I(t)&\leq&t^{-\alpha-1}\int_{\cB_r(a)}\left(\int_{t\cB_r(a)+(1-t)y}|u_x(w)|^p\nu_{\al}(dw)\right)\nu_{\al}(dy)\nonumber\\
&\le&
2^{\alpha+1}\int_{\cB_r(a)}\left(\int_{\cB_r(a)}|u_x(x)|^p\nu_{\al}(dx)\right)\nu_{\al}(dy)\nonumber\\
&=&2^{\alpha+1}\nu_{\al}(\cB_r(a))\int_{\cB_r(a)}|u_x(x)|^p\nu_{\al}(dx).\nonumber
\end{eqnarray}
 Now, (\ref{2010.03.17.1}) follows.
\end{proof}

  \begin{lemma}(\cite{KL11})
                        \label{2010.03.23.2}
Let $\alpha>0$. Denote $\nu_{\alpha}^1(dx^1)=(x^1)^{\alpha}dx^1$. For any $\cB^1_r(a):=(a-r,a+r)\subset \mathbb{R}_+$ we have a non-negative
function $\zeta\in C^{\infty}_0(\mathbb{R}_+)$ and a constant $N=N(\alpha)$ such that
\begin{equation}
                           \label{2011.12.10.1}
supp(\zeta)\in \cB^1_{r}(a), \quad \quad  \int_{\cB^1_r(a)}\zeta(x^1)\nu^1_{\alpha}(dx^1)=1,
\end{equation}
\begin{equation}
                                  \label{2010.03.19.1}
\sup_x\zeta \cdot \nu_{\alpha}^1(\cB^1_r(a))\le N,\quad \quad  \sup_x |\zeta_{x^1}|\cdot \nu_{\alpha}^1(\cB^1_r(a))\le
\frac{N}{r}.
\end{equation}

\end{lemma}

\begin{proof}
Choose a nonnegative smooth function $\psi=\psi(x^1)\in C^{\infty}_0(\cB^1_{1/2}(0))$ so that $\int_{\bR} \psi(x^1) dx^1=1$ and $\psi(x^1)=0$ for $|x^1|\geq 1/2$.
Define
$$
\zeta(x^1)=\frac{(x^1)^{-\alpha}}{r}\psi(\frac{x^1-a}{r}).
$$
Then  (\ref{2011.12.10.1}) is obvious. Since $r\leq a$ and $(a+r)^{\alpha+1}-(a-r)^{\alpha+1}\leq 2r(\alpha+1) (2a)^{\alpha}$,
\begin{eqnarray*}
\sup |\zeta| \cdot  \nu_{\alpha}^1(\cB^1_r(a)) &\leq &N\sup_{|x^1-a|\leq r/2} \frac{(x^1)^{-\alpha}}{r} \cdot ((a+r)^{\alpha+1}-(a-r)^{\alpha+1})\\
&\leq& N \frac{(a/2)^{-\alpha}}{r}\cdot ((a+r)^{\alpha+1}-(a-r)^{\alpha+1})\leq N.
\end{eqnarray*}
Similarly, the last inequality also holds because
\begin{eqnarray*}
\sup |\zeta_{x^1}| \cdot  \nu_{\alpha}^1(B^1_r(a)) &\leq& N\sup_{|x^1-a|\leq r/2} \left(\frac{(x^1)^{-\alpha}}{r^2}+ \frac{(x^1)^{-\alpha-1}}{r}\right)\cdot ((a+r)^{\alpha+1}-(a-r)^{\alpha+1})\\
&\leq& \frac{N}{r}(1 +  \frac{(2a)^{\alpha+1}}{(a/2)^{\alpha+1}})\leq \frac{N}{r}.
\end{eqnarray*}
The lemma is proved.
\end{proof}

Recall that for $t\in \bR$, $a\in\bR_+$ and  $x'\in \bR^{d-1} $
$$
\cQ_r(t,a,x'):=(t,t+r^2)\times (a-r,a+r)\times B'_r(x'), \quad  \cQ_r(a):=\cQ_r(0,a,0).
$$
{\underline{From this point on    we fix $\alpha:=\theta-d+p$}} and denote
$$
\nu^1(dx^1):=(x^1)^{\alpha}dx^1,  \quad   \quad \mu(dt dx)=\nu(dx)dt:=(x^1)^{\alpha}dxdt,
$$
$$
u_{\cQ_r(a)}=\frac{1}{\mu(\cQ_r(a))}\int_{\cQ_r(a)}u(t,x) \mu(dxdt).
$$

\begin{lemma}(\cite{KL11})
                              \label{key lemma 2}
Let $p\in [1,\infty)$,
$f^i,g\in C^{\infty}_{loc}(\Omega)$. Assume that $u\in C^{\infty}_{loc}(\Omega) $
satisfies the equation
\begin{eqnarray}
u_t+a^{ij}u_{x^ix^j}=f^i_{x^i}+g  \label{backward sys}
\end{eqnarray}
on $\cQ_r(a)\subset \Omega$. Then
\begin{eqnarray}
\int_{\cQ_r(a)}\left|u(t,x)-u_{Q_r(a)}\right|^p\mu(dtdx)\le N r^p
\int_{\cQ_r(a)}(|u_x(t,x)|^p+|f(t,x)|^p+r^p|g(t,x)|^p)\mu(dtdx),\label{2010.03.23.5}
\end{eqnarray}
where   $N=N(\theta,p,d,\delta,K)$.
\end{lemma}

\begin{proof}
We follow the outline
for the proof of Theorem 4.2.1 in \cite{kr08}.
We take the  function $\zeta$ corresponding to $\cB^1_r(a)$ and $\alpha$(:=$\theta-d+p$) from
Lemma \ref{2010.03.23.2}, and take a nonnegative function $\phi=\phi(x')\in C^{\infty}_0(B'_1(0))$ with unit integral. Denote $\eta(x')=r^{-d+1}\phi(x'/r)$, $\cB_r(a)=(a-r,a+r)\times B'_r(0)$ as before, and for $t\in (0,r^2)$ set
\begin{eqnarray}
\bar u(t):=\int_{\cB_r(a)}\zeta(y^1)\eta(y')u(t,y)\nu(dy).\nonumber
\end{eqnarray}
Then by Jensen's inequality and Poincar\'e's
inequality (Lemma \ref{2010.03.23.3}),
\begin{eqnarray}
&&\int_{\cB_r(a)}|u(t,x)-\bar u(t)|^p\nu(dx)\nonumber\\
&=&\int_{\cB_r(a)}\Big|\int_{\cB_r(a)}(u(t,x)-u(t,y))\zeta(y^1)\eta(y')\nu(dy)\Big|^p\nu(dx)\nonumber\\
&\le&
\int_{\cB_r(a)}\left(\int_{\cB_r(a)}|u(t,x)-u(t,y)|^p\zeta(y^1)\eta(y')\nu(dy)\right)\nu(dx)\nonumber\\
&\le&
|\sup\;\zeta|\cdot |\sup\,\eta|\,\int_{\cB_r(a)}\int_{\cB_r(a)}|u(t,x)-u(t,y)|^p\nu(dx)\nu(dy)\nonumber\\
&\le& N r^{-d+1}|\sup\;\zeta| \cdot \nu(\cB_r(a))\;r^p \int_{\cB_r(a)}|u_x(x)|^p \nu(dx)\nonumber\\
&\le& N r^{-d+1}|\sup\;\zeta| \cdot \nu^1((a-r,a+r))\; r^{d-1}r^p \int_{\cB_r(a)}|u_x(x)|^p \nu(dx)\nonumber\\
&\leq &N\;r^p\int_{\cB_r(a)}|u_x(x)|^p
\nu(dx). \label{2010.03.23.4}
\end{eqnarray}

\noindent
We observe that for any constant vector $c\in \mathbb{R}$ the
left-hand side of (\ref{2010.03.23.5}) is less than $2\cdot 2^p$
times
\begin{eqnarray}
                   \label{eqn 12.12.1}
\int_{\cQ_r(a)}|u(t,x)-c|^p\mu(dtdx)\le 2^p \int_{\cQ_r(a)}|u(t,x)-\bar
u(t)|^p\mu(dtdx)+2^p\;\nu(\cB_r(a))\int^{r^2}_{0}|\bar
u(t)-c|^pdt.
\end{eqnarray}
By (\ref{2010.03.23.4}) the first term  of the right side of (\ref{eqn 12.12.1}) is less than
(\ref{2010.03.23.5}). To estimate the second term, we take
$c=\frac{1}{r^2}\int^{r^2}_{0} \bar u(t)dt$.
Then by Poincar\'e's inequality without a weight in variable $t$ we
have
\begin{eqnarray}
\nu(\cB_r(a))\int^{r^2}_{0}|\bar u(t)-c|^pdt
\le N\;\nu(\cB_r(a))\;(r^2)^p
\int^{r^2}_{0}\Big|\int_{\cB_r(a)}\zeta(x^1)\eta(x')u_t(t,x)\nu(dx)\Big|^pdt. \label{2010.03.24.1}
\end{eqnarray}

\noindent
To estimate
the right side of (\ref{2010.03.24.1}), we recall  $u_t=-a^{ij}(t)u_{x^ix^j}+f^i_{x^i}+g$.  First, to handle the integral with $g$,  we use Jensen's inequality, take the supremum out of the integral to get
\begin{eqnarray}
&&\nu(\cB_r(a))\;r^{2p}
\int^{r^2}_{0}\Big|\int_{\cB_r(a)}\zeta(x^1)\eta(x')g(t,x)\nu(dx)\Big|^pdt\nonumber\\
&\le&\nu(\cB_r(a))\; r^{2p}\;|\sup\;\zeta|\,|\sup \eta|\,
\int^{r^2}_{0}\int_{\cB_r(a)}|g(t,x)|^p\nu(dx)dt\nonumber\\
&\le&N \nu^1((a-r,a+r))r^{d-1}\; r^{2p}\;|\sup\;\zeta|\,r^{-d+1}
\int^{r^2}_{0}\int_{\cB_r(a)}|g(t,x)|^p\nu(dx)dt\nonumber\\
&\le& N(\theta,p,d)\;r^{2p}\int_{\cQ_r(a)}|g(t,x)|^p\mu(dtdx),\nonumber
\end{eqnarray}
where we used $|\sup \zeta|\;\nu^1((a-r,a+r))\leq N$ (see Lemma \ref{2010.03.23.2}).

Next, we handle
 the integral with $-a^{ij}u_{x^ix^j}$. Fix  $i,j$.
 Firstly, assume either $i$ or $j$ is $1$; say $j=1$. We use integration by parts
and observe
\begin{eqnarray}
&&\nu(\cB_r(a))\;(r^2)^p
\int^{r^2}_{0}\Big|\int_{\cB_r(a)}\zeta(x^1)\eta(x')a^{ij}(t)u_{x^ix^j}(t,x)\nu(dx)\Big|^pdt\nonumber\\
&\le&\nu(\cB_r(a))\;r^{2p}\int^{r^2}_{0}\Big|\int_{\cB_r(a)}\zeta_{x^1}(x^1)\eta(x')a^{ij}(t)u_{x^i}(t,x)
\nu(dx)\Big|^pdt\nonumber\\
&&\quad+\nu(\cB_r(a))\;r^{2p}(\alpha-1)^p\int^{r^2}_{0}\Big|\int_{\cB_r(a)}\frac{1}{x^1}\zeta(x^1)\eta(x')a^{ij}(t)u_{x^i}(t,x)
\nu(dx)\Big|^pdt\nonumber\\
&=:& \quad I_1+I_2.\nonumber
\end{eqnarray}
For $I_2$ we use the fact $|a^{ij}u_{x^i}|\le |a^{ij}||u_{x^i}|\le K|u_x|$ and $1/{x^1}\leq 2/r$ on the support of $\zeta$.  The argument handling
the case of $g$ easily shows
\begin{eqnarray}
I_2\le N(K,\theta,p,d)\;r^{p}\int_{\cQ_r(a)}|u_x(t,x)|^p\mu(dtdx).\nonumber
\end{eqnarray}
For $I_1$ we use H\"older's inequality and get
\begin{eqnarray*}
\nu(\cB_r(a))\cdot|\int_{\cB_r(a)} \zeta_{x^1}\eta a^{ij}u_{x^i} \;d\nu|^p&\leq& \nu(\cB_r(a))^{p}\int_{\cB_r(a)}|\zeta_{x^1} \eta a^{ij}u_{x^i}|^p \;d\nu\\
&\leq& N(\nu^1((a-r,a+r))^p r^{(d-1)p}\cdot|\sup \zeta_{x^1}|^p r^{(-d+1)p}\int_{\cB_r(a)}|u_x|^p\nu(dx).
\end{eqnarray*}
Since $\nu^1((a-r,a+r))\cdot |\sup \zeta_x|\leq N/r$, it easily follows that
\begin{eqnarray}
I_1\le N(K,\theta,p,d)\;r^{p}\int_{\cQ_r(a)}|u_x(t,x)|^p\mu(dtdx).\nonumber
\end{eqnarray}
Secondly, if $i,j\neq 1$,  by integration by parts, H\"older's inequality and the inequality $\sup|\eta_{x'}|\leq N r^{-d}$,
\begin{eqnarray*}
&&\nu(\cB_r(a))\;r^{2p}\int^{r^2}_{0}\Big|\int_{\cB_r(a)}\zeta(x^1)\eta(x')\left[-a^{ij}(t)u_{x^ix^j}(t,x)\right]\nu(dx)\Big|^pdt\nonumber\\
&=&
\nu(\cB_r(a))\;r^{2p}\int^{r^2}_{0}\Big|\int_{\cB_r(a)}\zeta(x^1)\eta_{x^j}(x')a^{ij}(t)u_{x^i}(t,x) \nu(dx)\Big|^pdt\nonumber\\
&\leq&
\nu(\cB_r(a))^p\;r^{2p}\int^{r^2}_{0}\int_{\cB_r(a)}\Big|\zeta(x^1)\eta_{x^j}(x')a^{ij}(t)u_{x^i}(t,x)\Big|^p \nu(dx)\;dt\nonumber\\
&\leq&
N\nu(\cB_r(a))^p\;r^{2p}\cdot\sup |\zeta|^p\cdot r^{-dp}\int^{r^2}_{0}\int_{\cB_r(a)}|u_{x}|^p \nu(dx)dt\nonumber\\
&\leq&
N r^{p}\int_{\cQ_r(a)}|u_{x}|^p \mu(dxdt).
\end{eqnarray*}

For the integral with $f^i_{x^i}$ we use similar calculation to the one used to handle the term $-a^{ij}u_{x^ix^j}$, and get for each $i$
\begin{eqnarray}
&&\nu(\cB_r(a))\,r^{2p}\int^{r^2}_{0}\Big|\int_{\cB_r(a)}\zeta(x^1) \eta(x')f_{x^i}(t,x)\nu(dx)\Big|^pdt\nonumber\\
&\le& N \,r^{p}\int_{\cQ_r(a)}|f(t,x)|^p\mu(dtdx).\nonumber
\end{eqnarray}
Hence, the lemma is proved.
\end{proof}


\begin{lemma}(\cite{KL11})
                                   \label{key lemma 3}
Let $p\in [1,\infty)$,  $0<r\leq a$ and $u\in
C^{\infty}_{loc}(\Omega)$.

(i)  There is a constant $N=N(\theta,p,d,\delta,K)$
such that for any $\ell=1,\cdots,d$  we have
\begin{eqnarray}
                                  \label{101004.14.17}
\int_{\cQ_r(a)}\left|u_{x^{\ell}}(t,x)-(u_{x^{\ell}})_{\cQ_r(a)}\right|^p\mu(dtdx)\le N r^p
\int_{\cQ_r(a)}(|u_{xx}(t,x)|^p+|u_t(t,x)|^p)\mu(dtdx).
\end{eqnarray}

(ii) Denote $\kappa_0=\kappa_0(r,a):=(\nu^1((a-r,a+r))^{-1}\cdot \int^{a+r}_{a-r} x^1 \nu^1(dx^1)$. Then
\begin{eqnarray}
&&\int_{\cQ_r(a)}\left|u(t,x)-u_{\cQ_r(a)}+\kappa_0(u_{x^1})_{\cQ_r(a)}-\sum_{i=1}^d x^i(u_{x^i})_{\cQ_r(a)}\right|^p\mu(dtdx)\nonumber\\
&\le& N r^p\int_{\cQ_r(a)}
(|u_{x}(t,x)-(u_x)_{\cQ_r(a)}|^p+r^p|u_t(t,x)|^p+r^p|u_{xx}(t,x)|^p)\mu(dtdx)\nonumber\\
&\le& N
r^{2p}\int_{\cQ_r(a)}(|u_{xx}(t,x)|^p+|u_t(t,x)|^p)\mu(dtdx).  \label{101004.14.22}
\end{eqnarray}
\end{lemma}

\begin{proof}
(i) For (\ref{101004.14.17}) we use that fact that  $v:=u_{x^{\ell}}$ satisfies $v_t-a^{ij}v_{x^ix^j}=(f^i)_{x^i}$,  where $f^i=\delta^{i \ell}(u_t-a^{jm}u_{x^jx^m})$, and apply Lemma \ref{key lemma 2}.

(ii) To prove (\ref{101004.14.22}), denote
$v(t,x):=u(t,x)-(u)_{\cQ_r(a)}+\kappa_0(u_{x^1})_{Q_r(a)}-\sum_i x^i(u_{x^i})_{\cQ_r(a)}$.  Then
$$
v_{\cQ_r(a)}=\kappa_0(u_{x^1})_{\cQ_r(a)}-\sum_i\frac{(u_{x^i})_{\cQ_r(a)}}{\mu(\cQ_r(a))}\int_{\cQ_r(a)} x^i \nu(dx)dt=0,
$$
$$
 v-v_{\cQ_r(a)}=v, \quad v_{x^i}=u_{x^i}-(u_{x^i})_{\cQ_r(a)}, \quad v_t-a^{ij}v_{x^ix^j}=g:=u_t-a^{ij}u_{x^ix^j}.
 $$
  Now it is enough to use Lemma \ref{key lemma 2} and (\ref{101004.14.17}). The lemma is proved.
\end{proof}


\begin{theorem}\label{101004.16.53}
Let $\theta\in (d-1, d-1+p)$,  $0<r\leq a$ and $\nu r/a\geq 2$.  Assume that $u\in
C^{\infty}_{loc}(\Omega)$ satisfies $u_t+a^{ij}(t)u_{x^ix^j}=0$ in $\cQ_{\nu
r}(t_0,a,x'_0)\cap\Omega$. Then there is a constant $N=N(K,\delta,\theta,p,d)$
so that
\begin{eqnarray}
&& \aint_{\cQ_r(t_0,a,x'_0)}|u_{xx}(t,x)-(u_{xx})_{\cQ_r(t_0,a,x'_0)}|^p \mu(dt dx) \nonumber\\
&\leq&
\frac{N}{(1+\nu r/a)^p}\;\aint_{\cQ_{\nu r}(t_0,a,x'_0)\cap
\Omega}|u_{xx}(t,x)|^p \mu(dt dx).\label{eqn 2.21.2011}
\end{eqnarray}
\end{theorem}

\begin{proof}
Considering a proper translation, without loss of generality, we assume that  $t_0=0$, $x'_0=0$ and thus $\cQ_r(t_0,a,x'_0)=\cQ_r(a)$.

{\bf{Step 1}}. First, we consider the case $a=1$.  Obviously,
$$
r\leq 1, \quad 2\leq \nu r, \quad \beta:=\frac{1+\nu r}{2} \leq \nu r, \quad \frac{r}{\beta}\leq \frac{1}{\beta}\leq \frac{2}{3}, \quad \quad 2\beta =1+\nu r.
$$
Thus,
$$
 \cQ_{\beta}(\beta) \subset \cQ_{\nu r}(1)\cap \Omega, \quad \cQ_{r/\beta}(\beta^{-1}) \subset \cQ_{2/3}(2/3).
 $$
 Denote $w(t,x)=u(\beta^2t,\beta x)$, then
 obviously
 $$
w_t+a^{ij}(\beta^2t)w_{x^ix^j}=0, \quad \quad \text{for}\quad (t,x)\in \cQ_1(1)
$$
and
\begin{eqnarray*}
\aint_{\cQ_r(1)}|u_{xx}(t,x)-(u_{xx})_{\cQ_r(1)}|^p \mu(dtdx)& \leq& N(d)\sup_{\cQ_r(1)}(|u_{xxx}|^p+|u_{xxt}|^p)\\
&\leq& N(d)\beta^{-3p} \,\,\sup_{\cQ_{r/\beta}(\beta^{-1})}(|w_{xxx}|^p+|w_{xxt}|^p)\\
&\leq& N(d)\beta^{-3p} \,\,\sup_{\cQ_{2/3}(2/3)} (|w_{xxx}|^p+|w_{xxt}|^p).
\end{eqnarray*}
Applying  Lemma \ref{101004.14.53} to
$v(t,x)=w(t,x)-w_{\cQ_1(1)}+\kappa_0(w_{x^1})_{\cQ_1(1)}-\sum_{i=1}^d x^i(w_{x^i})_{\cQ_1(1)}$, and then using   Lemma \ref{key lemma 3}
\begin{eqnarray*}
\beta^{-3p} \,\,\sup_{\cQ_{2/3}(2/3)} (|w_{xxx}|^p+|w_{xxt}|^p) &\leq& N\beta^{-3p}\int_{\cQ_1(1)}|v|^p \mu(dtdx)\\
&\leq& N\beta^{-3p}\int_{\cQ_1(1)}|w_{xx}|^p \mu(dtdx)\\
&=&N\beta^{-2p-2-\theta}\int_{\cQ_{\beta}(\beta)}|u_{xx}|^p \mu(dtdx).
\end{eqnarray*}
This leads to (\ref{eqn 2.21.2011}) since $|\cQ_{\nu r}(1)\cap \Omega|\sim \beta^{p+\theta+2}$.

{\bf{Step 2}}. Let $a\neq 1$. Define $v(t,x):=u(a^2t,ax)$. Then $v_t+a^{ij}(a^2t)v_{x^ix^j}=0$ in $\cQ_{\nu
r/a}(1)\cap\Omega$. It is easy to check
$$
\mu(\cQ_{r/a}(1))=a^{-\theta-p-2}\mu(\cQ_r(a)), \quad (v_{xx})_{\cQ_{r/a}(1)}=a^2(u_{xx})_{\cQ_r(a)}, \quad \mu(\cQ_{\nu r/a}(1)\cap \Omega)=a^{-\theta-p-2}\mu(
\cQ_{\nu r}(a)\cap \Omega),
$$
and consequently
$$
\aint_{\cQ_{r/a}(1)}|v_{xx}(t,x)-(v_{xx})_{\cQ_{r/a}(1)}|^p\mu(dtdx)=a^{2p}\aint_{\cQ_{r}(a)}|u_{xx}(t,x)-(u_{xx})_{\cQ_{r}(a)}|^p\mu(dtdx),
$$
$$
\aint_{\cQ_{\nu r/a}(1)\cap
\Omega}|v_{xx}(t,x)|^p\mu(dtdx)=a^{2p}\aint_{\cQ_{\nu r}(a)\cap
\Omega}|u_{xx}(t,x)|^p\mu(dtdx).
$$
It follows that
\begin{eqnarray*}
\aint_{\cQ_{r}(a)}|u_{xx}(t,x)-(u_{xx})_{\cQ_{r}(a)}|^p\mu(dtdx)&=& a^{-2p}\aint_{\cQ_{\nu r/a}(1)\cap
\Omega}|v_{xx}(t,x)|^p\mu(dtdx)\\
&\leq& a^{-2p} \frac{N}{(1+\nu r/a)^p}\;\aint_{\cQ_{\nu r/a}(1)\cap
\Omega}|v_{xx}(t,x)|^p\mu(dtdx)\\
&=&\frac{N}{(1+\nu r/a)^p}\;\aint_{\cQ_{\nu r}(a)\cap
\Omega}|u_{xx}(t,x)|^p\mu(dtdx).
\end{eqnarray*}
The theorem is proved.

\end{proof}

The following is the main result of this section. Recall that $\theta<d-1+p$. Thus for  $q$ sufficiently close to $p$, we have $\theta+p-q<d-1+q$.

\begin{theorem}
                                  \label{theorem 5.2.1}
Let $\theta\in (d-1, d-1+p)$,  $0<r\leq a$ and $p,q\in (1,\infty)$ so that
\begin{equation}
                            \label{eqn 5.02.2}
q\leq p,\quad \quad \theta':=\theta+p-q<d-1+q.
\end{equation}
 Also let  $\nu\ge 2$,  $r\nu\geq a$ and   $u\in
C^{\infty}(\Omega)$. Then,
\begin{eqnarray*}
&&\aint_{\cQ_r(t_0,a,x_0)}|u_{xx}(t,x)-(u_{xx})_{\cQ_r(t_0,a,x_0)}|^q \mu(dtdx)\\
&\le& N
\frac{1}{(1+\nu r/a)^q}\aint_{\cQ_{\nu r}(t_0,a,x'_0)\cap
\Omega}|u_{xx}(t,x)|^q\mu(dtdx) \\
&+&N \frac{\nu^{d+1}}{r/a}(1+\nu r/a)^{p+\theta-d+1}\aint_{\cQ_{\nu r}(t_0,a,x'_0)\cap
\Omega}|u_t+a^{ij}u_{x^ix^j}|^q\mu(dtdx),
\end{eqnarray*}
where $N=N(K,\delta,\theta,p,q)$.

\end{theorem}
\begin{proof}
 As before we may assume that $t_0=0$ and $x'_0=0$. Also we may assume that $a^{ij}(t)$ is infinitely differentiable in $t$ and all the derivatives of $a^{ij}$ are bounded.  Indeed, take a sequence of smooth functions $a_{n}^{ij}$
so that each $a^{ij}_{n}$ satisfies condition (\ref{assumption 1}) and $a^{ij}_n(t)\to a^{ij}(t)$ as $n\to \infty$ (a.e.).  Then it is enough to observe
$$
\aint_{\cQ_{\nu r}(t_0,a,x'_0)\cap
\Omega}|u_t+a_nu_{xx}|^q\mu(dtdx) \,\, \to \,\,\aint_{\cQ_{\nu r}(t_0,a,x'_0)\cap
\Omega}|u_t+au_{xx}|^q\mu(dtdx) \quad \quad \text{as}\quad n\to \infty.
$$
Also note that we may assume that $u(t)$ vanishes for all large $t$, say for all $t\geq T \,(\geq \nu^2 r^2)$.

Take a $\zeta\in C^{\infty}_0(\bR^{d+1})$
so that $\zeta(t,x)=1$ for $(t,x)\in \cQ_{\nu r/2}(a)\cap \Omega$ and  $\zeta(t,x)=0$ if $(t,x)\not\in (-\nu^2r^2,\nu^2r^2)\times (-a,a+\nu r)\times B'_{\nu r}$.  Denote
$$
f=u_t+a^{ij}u_{x^ix^j}, \quad g=f\zeta, \quad h=f(1-\zeta).
$$
By Corollary  \ref{cor 5.1.1} we can define  $v$ as the solution of
\begin{equation}
                       \label{eqn again}
v_t+a^{ij}v_{x^ix^j}=h \quad \text{for}\,\, t\in (-\infty,T), \quad \text{and}\quad  v(T,\cdot)=0
\end{equation}
so that $v\in M\bH^n_{p,\theta}(-\infty,T)$ for any  $n$.  Also let $\bar{v}\in  M\bH^n_{p,\theta}(-\infty,T+1)$ be the solution of
$$
\bar{v}_t+a^{ij}\bar{v}_{x^ix^j}=h \quad \text{for}\,\, \quad t\in (-\infty,T+1), \quad \text{and}\quad  \bar{v}(T+1,\cdot)=0.
$$
Then by considering the equation for $\bar{v}$  on $(T,T+1)$, since $h(t)=0$ for $t\geq T$, we conclude $\bar{v}(t)=0$ for $t\in [T,T+1]$. Thus $\bar{v}$ also satisfies (\ref{eqn again}) and $v=\bar{v}$.
It follows from (\ref{eqn 3.31.4}) that $v$ is infinitely differentiable in $x$ (and hence in $t$), and thus $v\in C^{\infty}_{loc}(\Omega)$.

By (\ref{eqn 5.02.2}),
$$
\theta-d+p=\theta'-d+q, \quad \theta'\in (d-1, d-1+q).
$$
 By applying Theorem \ref{101004.16.53} with $q,\theta'$ and $\nu/2$ in places of $p,\theta$ and $\nu$ respectively,
\begin{eqnarray}
               \aint_{\cQ_r(a)}|v_{xx}(t,x)-(v_{xx})_{\cQ_r(a)}|^q \bar{\mu}(dyds) &\le& N
\frac{1}{(1+\nu r/2a)^q}\aint_{\cQ_{\nu r/2}(a)\cap \Omega}|v_{xx}(t,x)|^q\bar{\mu}(dyds) \nonumber\\
&\le& N
\frac{1}{(1+\nu r/a)^q}\aint_{\cQ_{\nu r}(a)\cap \Omega}|v_{xx}(t,x)|^q\bar{\mu}(dyds),  \label{eqn 6.08.2}
\end{eqnarray}
where $\bar{\mu}(dsdy):=(y^1)^{\theta'-d+q}dyds=\mu(dyds)$.
On the other hand, $w:=u-v$ satisfies
$$
w_t+a^{ij}w_{x^ix^j}=g, \quad t\in (0,T).
$$
and  $w(T,\cdot)=0$.    By Corollary  \ref{cor 5.1.1} (with $q,\theta'$ in place of $p,\theta$ respectively),
$$
\int_{\cQ_r(a)}|w_{yy}|^q (y^1)^{\theta'-d+q}dyds\leq \int_{\cQ_{\nu r}(a)\cap \Omega} |w_{yy}|^q \mu(dsdy) \leq N\int_{\cQ_{\nu r}(a)\cap \Omega}|f|^q\mu(dsdy),
$$
\begin{eqnarray}
\aint_{\cQ_r(a)}|w_{yy}|^q\mu(dyds)&\leq& N\frac{\nu^{d+1}(1+\nu r/a)^{p+\theta-d+1}}{(1+r/a)^{p+\theta-d+1}-(1-r/a)^{p+\theta-d+1}}\aint_{\cQ_{\nu r}(a)\cap \Omega}|f|^q\mu(dyds) \nonumber\\
&\leq &N \frac{a}{r}\nu^{d+1}(1+\nu r/a)^{p+\theta-d+1}\aint_{\cQ_{\nu r}(a)\cap \Omega}|f|^q\mu(dyds), \label{eqn 6.08.1}
\end{eqnarray}
where  inequality (\ref{eqn 6.08.1}) is obtained as follows; since $p+\theta-d+1\geq 1$,
$$
(1+r/a)^{p+\theta-d+1}-(1-r/a)^{p+\theta-d+1}\geq (1+r/a) - (1-r/a)\geq 2r/a.
$$
Observing that $u=v+w$, we get
\begin{eqnarray*}
I:&=&\aint_{\cQ_r(a)}|u_{yy}(t,x)-(u_{yy})_{\cQ_r(a)}|^q \mu(dyds)\\
&\leq& N(q)\aint_{\cQ_r(a)}|w_{yy}(t,x)-(w_{yy})_{\cQ_r(a)}|^q \mu(dyds)+N(q)\aint_{\cQ_r(a)}|v_{yy}(t,x)-(v_{yy})_{\cQ_r(a)}|^q \mu(dyds)\\
&\leq& N(q)\aint_{\cQ_r(a)}|w_{yy}(t,x)|^q \mu(dyds)+N(q)\aint_{\cQ_r(a)}|v_{yy}(t,x)-(v_{yy})_{\cQ_r(a)}|^q \mu(dyds).
\end{eqnarray*}
Thus  by (\ref{eqn 6.08.2}) and (\ref{eqn 6.08.1}),
\begin{eqnarray*}
I  &\leq&N\frac{a}{r}\nu^{d+1}(1+\nu r/a)^{p+\theta-d+1}\aint_{\cQ_{\nu r}(a)\cap \Omega}|f|^q\mu(dyds) + N
\frac{1}{(1+\nu r/a)^q}\aint_{\cQ_{\nu r}(a)\cap \Omega}|v_{yy}(t,x)|^q\mu(dyds)\\
&\leq& N\frac{a}{r}\nu^{d+1}(1+\nu r/a)^{p+\theta}\aint_{(0,\nu^2r^2)\times(0,a+\nu r)}|f|^q\mu(dyds)\\
&&+ N\frac{1}{(1+\nu r/a)^q}\aint_{\cQ_{\nu r}(a)\cap \Omega}\left(|u_{yy}(t,x)|^q+|w_{yy}(t,x)|^q\right)\mu(dyds)\\
&\leq& N\frac{a}{r}\nu^{d+1}(1+\nu r/a)^{p+\theta-d+1}\aint_{\cQ_{\nu r}(a)\cap \Omega}|f|^q\mu(dyds)
+ N\frac{1}{(1+\nu r/a)^q}\aint_{\cQ_{\nu r}(a)\cap \Omega}|u_{yy}(t,x)|^q\mu(dyds).
\end{eqnarray*}
The theorem is proved.
\end{proof}

\mysection{A priori estimate for equations with  BMO coefficients}
                                     \label{section sharp 2}

Recall
that    $\cQ_r(t,x)=(t,t+r^2) \times (x^1-r,x^1+r)\times B'_r(x')$.  For any $d\times d$ matrix $a=(a^{ij}(t,x))$, as in \cite{Kr07}, we   define a standard mean oscillation on $\cQ_r(t,x)=\cQ_r(t,x^1,x')$:
\begin{eqnarray}
                             \label{oscillation}
osc_x(a,\cQ_r(t,x))=\frac{1}{r^2|\cB_r(x)|^2} \int^{t+r^2}_t \left(\int_{\cB_r(x)}\int_{\cB_r(x)} |a(s,y)-a(s,z)|dydz\right)ds,
\end{eqnarray}
where $|\cB_r(x)|$ is the Euclidian volume of $\cB_r(x)$.  We say that {\underline{$a$ is $VMO$}} (see \cite{Kr07} for more details)  if
\begin{equation}
                          \label{eqn 12.12.6}
\lim_{r\to 0} \sup_{\cQ_r(t,x)} osc_x(a,\cQ_r(t,x))=0.
\end{equation}

Now we define a   mean oscillation with respect to measure $\nu(dx)=(x^1)^{\theta-d+p}dx$:
\begin{eqnarray*}
osc^{\theta}_x(a,\cQ_r(t,x^1,x'))&=&r^{-2}\int^{t+r^2}_t \left(\aint_{\cB_r(x)}\aint_{\cB_r(x)} |a(s,y)-a(s,z)|\nu(dy)\nu(dz)\right)ds\\
&=&\frac{1}{r^2(\nu(\cB_r(x)))^2}\int^{t+r^2}_t \left(\int_{\cB_r(x)}\int_{\cB_r(x)} |a(s,y)-a(s,z)|\nu(dy)\nu(dz)\right)ds.
\end{eqnarray*}

\noindent
 Obviously, if $a$ depends only on $t$ then $osc^{\theta}_x(a,\cQ_r(t,x^1,x'))=0$. Also it is easy to check that for any $d\times d$ matrix-valued $\bar{a}(t)$ depending only on $t$,
$$
osc^{\theta}_x(a,\cQ_r(t,x))\leq 2r^{-2}\int^{t+r^2}_t\aint_{\cB_r(x)}|a(s,y)-\bar{a}(s)|\nu(dy)ds.
$$
On the other hand,
$$
r^{-2}\int^{t+r^2}_t\aint_{\cB_r(x)}|a(s,y)-(a)_{\cB_r(x)}(s)|\nu(dy)ds \leq osc^{\theta}_x(a,\cQ_r(t,x)),
$$
where $(a)_{\cB_r(x)}(s)=(\nu(\cB_r(x)))^{-1}\int_{\cB_r(x)} a(s,y)\nu(dy)$.

Roughly speaking, the following result says that the condition $osc^{\theta}_x(a,\cQ_r(t,x^1,x'))\leq \varepsilon$ for some $\varepsilon$ is not
stronger than the condition $osc_x(a,\cQ_r(t,x^1,x'))\leq \varepsilon$.

\begin{lemma}
                                     \label{lemma inverse}
There exists a constant $N=N(\theta)>0$ so that for any $\kappa\in (0,1]$ and $r=\kappa x^1$,
\begin{equation}
                    \label{eqn 5.02.8}
osc^{\theta}_x(a,\cQ_r(t,x^1,x'))\leq N \, \,osc_x(a,\cQ_r(t,x^1,x')),
\end{equation}
\begin{equation}
                       \label{eqn 10.2.11}
  osc_x(a,\cQ_r(t,x^1,x'))\leq N\cdot (1-\kappa)^{-\alpha} \, \,osc^{\theta}_x(a,\cQ_r(t,x^1,x')).
  \end{equation}
  \end{lemma}


\begin{proof}
Denote $\alpha:=\theta-d+p>-1$. First note that
$$ \nu(dy)\leq (x^1)^{\alpha}(1+\kappa)^{\alpha}dy \quad \quad \text{on}\quad \cB:=\cB_r(x),
$$
$$
\frac{|\cB|}{\nu(\cB)}=(\alpha+1)(x^1)^{-\alpha}\frac{2\kappa}{\left[(1+\kappa)^{\alpha+1}-(1-\kappa)^{\alpha+1}\right]}\leq N(\alpha)(x^1)^{-\alpha},
$$
where the last inequality is obtained as in (\ref{eqn 6.08.1}).  Thus
\begin{eqnarray*}
osc^{\theta}_x(a,\cQ_r(t,x^1,x'))&=&\frac{|\cB|^2}{(\nu(B))^2} \frac{r^{-2}}{|\cB|^2}\int^{t+r^2}_t\left(\int_{\cB}\int_{\cB} |a(s,y)-a(s,z)|\nu(dy)\nu(dz)\right)ds\\
&\leq&N^2(\alpha)(x^1)^{-2\alpha}\cdot (x^1)^{2\alpha}(1+\kappa)^{2\alpha}\,osc_x(a,\cQ_r(t,x^1,x'))\\
&\leq& N \,osc_x(a,\cQ_r(t,x^1,x')).
\end{eqnarray*}
\noindent
To prove (\ref{eqn 10.2.11}) it is enough to assume $\kappa \in (0,1)$ and note $dy\leq \frac{\nu(dy)}{(x^1)^{\alpha}(1-\kappa)^{\alpha}}$ on $\cB_r(x)$. The lemma is proved.
\end{proof}

\begin{remark}
                   \label{main remark}
In Theorem \ref{main theorem},  the  following condition near  $\partial \bR^d_+$ is assumed:
\begin{equation}
                        \label{12.12.3}
\lim_{x^1 \to 0} \sup_{r \leq \kappa_0 x^1} osc^{\theta}_x(a, \cQ_r(t,x^1,x')) <\varepsilon,
\end{equation}
where $\kappa_0, \varepsilon \in (0,1)$ will be specified later.
To understand (\ref{12.12.3}), let $d=1$  (so that $\bR^d_+=(0,\infty)$) and $a$ be independent of $t$. Then (\ref{12.12.3}) becomes
\begin{equation}
                     \label{eqn 12.12.5}
\lim_{x^1 \to 0} \sup_{r \leq \kappa_0 x^1} osc^{\theta}_x(a, (x^1-r,x^1+r)) <\varepsilon.
\end{equation}
Since $x^1-r \geq (1-\kappa_0)x^1>0$,  there is no requirement that the mean oscillation on a ball containing the boundary points is small.
{\bf{Obviously (\ref{12.12.3}) is satisfied if $a$ is $VMO$}}, since  (cf. (\ref{eqn 5.02.8}))
$$
\lim_{x^1 \to 0} \sup_{r \leq \kappa_0 x^1} osc^{\theta}_x(a, \cQ_r(t,x^1,x'))\leq N\lim_{x^1\to 0} \sup_{r\leq x^1} osc_x (a, \cQ_r(t,x))=0.
$$
\end{remark}


\vspace{4mm}

Note that in the following result $dxdt$ is used in place of $\mu(dxdt)$. However, if  $r/a$ is  small then  $\aint_{\cQ_r(a)} dtdx$ and $\aint_{\cQ_r(a)} \mu(dtdx)$ are comparable.

\begin{lemma}
                      \label{sharp entire}
Let $q>1$ and $a^{ij}=a^{ij}(t)$. Then there exists a constant $N=N(\delta,K,p,d)$ so that for any $\nu\geq 4$, $r>0$ and $u\in C^{\infty}(\Omega)$,
$$
\aint_{\cQ_r(a)}\aint_{\cQ_r(a)}|u_{xx}(t,x)-u_{xx}(s,y)|^q dxdtdyds\leq N\nu^{-q}\aint_{\cQ_{\nu r}(a)}|u_{xx}|^q dxdt +N\nu^{d+2}\aint_{\cQ_{\nu r}(a)}|u_t+a^{ij}u_{x^ix^j}|^q dxdt
$$
\end{lemma}
\begin{proof}
See Theorem 6.1.2 of \cite{kr08}.
\end{proof}


For $\kappa \in (0,1]$ and $R>0$, let $\bQ(R,\kappa)$ be the collection of all $\cQ_r(t,x)$ so that  $r\leq \kappa x^1$ and $\cQ_r(t,x)\subset \{(t,y)\in \Omega
: y^1\in (0,R)\}$. That is, $\cQ_r(t,x)\in \bQ(R,\kappa)$ if
$$
x^1>0, \quad r\leq \kappa x^1, \quad x^1+r\leq R.
$$
Define
$$
a^{\#(\theta)}_{R,\kappa}=\sup_{\cQ\in \bQ(R,\kappa)} \, osc^{\theta}_x(a,\cQ), \quad \quad a^{\#(\theta)}_{\kappa}=\sup_{R>0}a^{\#(\theta)}_{R,\kappa}.
$$

\begin{lemma}
                                   \label{lemma 5.4.01}
      Let $\beta   \in (1,\infty)$, $\kappa \in (1/2,1)$ and
      $$
1<q< p,\quad \theta+p-q<d-1+q.
      $$
        Suppose that  $u\in C^{\infty}(\Omega)$ vanishes outside $\cQ_0 \in \bQ(R,\kappa)$.   Then for any $\varepsilon>0$,
     $\cQ_r(t_1,a,x'_1)\subset\Omega$ and $(t,x)\in \cQ_r(t_1,a,x'_1)$ we have
\begin{eqnarray}
&&\aint_{\cQ_r(t_1,a,x'_1)}|u_{xx}-(u_{xx})_{\cQ_r(t_1,a,x'_1)}|^q\mu(dyds) \nonumber\\
 &\leq& \varepsilon \bM(|u_{xx}|^q)(t,x)+N\, \bM(|f|^q)(t,x)
+N\, (a^{\#(\theta)}_{2R,\kappa})^{1/\beta'}\cdot\bM^{1/\beta}(|u_{xx}|^{\beta q})(t,x) \label{5.4.01}
\end{eqnarray}
where $f:=u_t+a^{ij}u_{x^ix^j}$, $\beta':=\beta/(\beta-1)$ and $N=N(\varepsilon,\theta,q,d,\delta,K)$.
        \end{lemma}

\begin{proof}

Let $\cQ_0=\cQ_{r_0}(t_0,a_0,x'_0)$.   Considering a translation,  we may assume $t_1=0$, $x'_1=0\in \mathbb{R}^{d-1}$ so that $\cQ_{r}(t_1,a,x'_1)=\cQ_{r}(a)$. Also, we assume
      \begin{equation}
                                     \label{eqn 5.9.1}
      \cQ_r(a) \cap \cQ_0\neq \emptyset.
      \end{equation}
Otherwise, the left term of (\ref{5.4.01}) becomes zero.

{\bf{Step 1}}. Firstly, we prove that there exists $\delta_0=\delta_0(\varepsilon)\in (0,1)$ so that (\ref{5.4.01})  holds if  $r/a\leq \delta_0$.
Let $|\cQ|$ denote the Lebesgue measure of $\cQ\subset\bR^{d+1}$.  Assume  $\nu\geq 4$ and $\nu r\leq a/4$. Then  $(3a/4) \leq x^1\leq (5a/4)$ if $x^1\in \cB^1_{\nu r}(a):=(a-\nu r, a+\nu r)$. Denote $c_0:=\left(\frac53\right)^{\theta-d+p}$, then
 $$
 \frac{\mu(dtdx)}{\mu(\cQ_r(a))}\leq c_0  \frac{dtdx}{|\cQ_r(a)|} \quad \quad \text{on}\,\,\, \cQ_r(a),
 $$
 $$
 \frac{dtdx}{|\cQ_{\nu r}(a)|}\leq c_0  \frac{\mu(dtdx)}{\mu(\cQ_{\nu r}(a))} \quad \quad \text{on}\,\,\, \cQ_{\nu r}(a).
 $$
 Also due to (\ref{eqn 5.9.1}), we have $a-r<a_0+r_0$ and thus
 \begin{equation}
                     \label{eqn 5.9.4}
 a+\nu r\leq 2R,  \quad \frac{\nu r}{a}\leq 1/4\leq \kappa, \quad \cQ_{\nu r}(a)\subset \bQ_{2R,\kappa}.
 \end{equation}
 Denote $\bar{a}^{ij}(t)=(a^{ij}(t,\cdot))_{B_{\nu r}(a)}$ and $f=u_t+a^{ij}u_{x^ix^j}$.  By Lemma \ref{sharp entire},
 \begin{eqnarray}
 &&\aint_{\cQ_r(a)}|u_{xx}-(u_{xx})_{\cQ_r(a)}|^q\mu(dsdy)\nonumber\\
  &\leq& \frac{1}{(\mu(\cQ_r(a)))^2}\int_{\cQ_r(a)} \int_{\cQ_r(a)}|u_{xx}(s,y)-u_{xx}(\tau,\xi)|^q \mu(dsdy)\mu(d\tau d\xi) \nonumber\\
 &\leq& c^2_0\frac{1}{|\cQ_r(a)|^2}\int_{\cQ_r(a)}\int_{\cQ_r(a)}|u_{xx}(s,y)-u_{xx}(\tau,\xi)|^qdsdy\,d\tau d\xi\\
 &\leq& Nc^2_0\nu^{d+2}\int_{\cQ_{\nu r}(a)}|u_t+\bar{a}^{ij}u_{x^ix^j}|^q \frac{dyds}{|\cQ_{\nu r}(a)|} +Nc^2_0\nu^{-q}\int_{\cQ_{\nu r}(a)}|u_{xx}|^q \frac{dyds}{|\cQ_{\nu r}(a)|}\nonumber\\
 &\leq&Nc^3_0\nu^{d+2}\aint_{\cQ_{\nu r}(a)}|f|^q \mu(dyds) +Nc^3_0\nu^{d+2} \cdot J    + Nc^3_0\nu^{-q}\aint_{\cQ_{\nu r}(a)}|u_{xx}|^q \mu(dyds),\label{eqn 5.5.7}
 \end{eqnarray}
where $N=N(d,\delta,K)$ and
\begin{eqnarray*}
J := \aint_{\cQ_{\nu r}(a)}|(a^{ij}-\bar{a}^{ij})u_{x^ix^j}|^q \,\mu(dtdx)\leq NJ^{1/\beta}_1 J^{1/\beta'}_2,
\end{eqnarray*}
$$
J_1:=\aint_{\cQ_{\nu r}(a)}|u_{xx}|^{q\beta}\mu(dtdx)  \leq N  \bM(|u_{xx}|^{\beta q})(t,x),
$$
\begin{eqnarray}
J_2:=\aint_{\cQ_{\nu r}(a)}|a^{ij}-\bar{a}^{ij}|^{q\beta'}\mu(dtdx)&\leq& N \aint_{\cQ_{\nu r}(a)}|a^{ij}-\bar{a}^{ij}|\mu(dtdx) \label{eqn 5.5.6}\\
&\leq& N a^{\#(\theta)}_{2R,\kappa}, \label{eqn 5.9.6}
\end{eqnarray}
where inequality (\ref{eqn 5.5.6}) is due to  $|a^{ij}|\leq K$, and (\ref{eqn 5.9.4}) is used in (\ref{eqn 5.9.6}).  Coming back to (\ref{eqn 5.5.7}), we get
\begin{eqnarray*}
&&\aint_{\cQ_r(a)}|u_{xx}-(u_{xx})_{\cQ_r(a)}|^q\mu(dsdy)\\
&\leq& N\nu^{d+2}\bM(|f|^q)(t,x)+N\nu^{-q}\bM(|u_{xx}|^q)(t,x)+N\nu^{d+2}(a^{\#(\theta)}_{2R,\kappa})^{1/\beta'}\bM^{1/\beta}(|u_{xx}|^{q\beta})(t,x).
\end{eqnarray*}
Remember that the above inequality holds whenever $\nu\geq 4$ and $r/a\leq (4\nu)^{-1}$. Now we fix $\nu$ so that $N\nu^{-q}\leq \varepsilon$ and take $\delta_0=1/{(4\nu)}$. Then whenever
$r/a\leq \delta_0$ we have $(r/a)\nu\leq 1/4$ and thus (\ref{5.4.01})  follows.

{\bf{Step 2}}.
 For given $\varepsilon$, take  $\delta_0=\delta_0(\varepsilon)$ from Step 1. Assume $r/a\geq  \delta_0$. Choose $\nu$, which will be specified later, so that $r\nu>4a$. Denote $\alpha:=\theta-d+p$.

 Here we claim  that if $\mu(\cQ_0)\geq 2^{d+2\alpha+3}\mu(\cQ_{\nu r}(a)\cap \Omega)$, then for $\bar{a}:=x_0-r_0+\nu r$ we have
 \begin{equation}
                     \label{eqn 5.10.1}
 \left(\cQ_{r_0}(t_0,x_0,x'_0)\cap \cQ_{\nu r}(a)\right)\subset \cQ_{\nu r}(\bar{a}), \quad \bar{a}+\nu r\leq 2R,\quad \nu r/{\bar{a}}\leq \kappa, \quad |\cQ_{\nu r}(\bar{a})|\leq 2^{\alpha+1}|\cQ_{\nu r}(a)\cap \Omega|.
 \end{equation}
First,  due to (\ref{eqn 5.9.1}), we have $0< x_0-r_0<2a$.  Let $\omega_{d-1}$ denote the volume of $B'_1(0)$. Then
$$
\mu(\cQ_0)\leq |\cQ_{x_0}(t_0,x_0,x'_0)|=\frac{1}{\alpha+1}\omega_{d-1}2^{\alpha+1}x^{\alpha+d+2}_0,
$$
$$
|\cQ_{\nu r}(a)\cap \Omega|=\frac{1}{\alpha+1}\omega_{d-1}(a+\nu r)^{\alpha+1} (\nu r)^{d+1} \geq \frac{1}{\alpha+1}\omega_{d-1}(\nu r)^{\alpha+d+2}.
$$
 Thus, by assumption it follows that  $2^{\alpha+1}x_0^{d+\alpha+2}\geq 2^{d+2\alpha+3}(\nu r)^{\alpha+d+2}$ or equivalently $x_0\geq 2\nu r$. Observe
 $$
 \left(\cQ_{r_0}(t_0,x_0,x'_0)\cap \cQ_{\nu r}(a)\right)\subset \left((0, (\nu r)^2)\times (x_0-r_0, a+\nu r)\times B'_{\nu r}\right)\subset \cQ_{\nu r}(\bar{a}).
 $$
 Also from the inequality $r_0\geq 1/2 x_0\geq \nu r$ (recall $\kappa\geq 1/2$), we get 
 $$
\frac{\nu r}{\bar{a}}=\frac{\nu r}{x_0-r_0+\nu r}\leq \frac{r_0}{x_0}\leq \kappa.
$$
Since the last inequality of (\ref{eqn 5.10.1}) is obvious, the claim is proved. Note that (\ref{eqn 5.10.1})  implies that $\cQ_{\nu r}(\bar{a})\in \bQ_{2R,\kappa}$.

Now define $\bar{a}^{ij}=(a^{ij})_{\cQ_{r_0}(t_0,x_0,x'_0)}$ if $|\cQ_{r_0}(t_0,x_0,x'_0)|< 2^{d+2\alpha+3}|\cQ_{\nu r}(a)\cap \Omega|$, and  otherwise define
$\bar{a}^{ij}=(a^{ij})_{\cQ_{\nu r}(\bar{a})}$, where $\bar{a}=x_0-r_0+\nu r$ as defined above. By Theorem \ref{theorem 5.2.1}
\begin{eqnarray*}
&&\aint_{\cQ_r(a)}|u_{xx}(t,x)-(u_{xx})_{\cQ_r(a)}|^q \mu(dtdx)\\
&\le&
\frac{N}{(1+\nu r/a)^q}\aint_{\cQ_{\nu r}(a)\cap
\Omega}|u_{xx}(t,x)|^q\mu(dtdx)
+\frac{N \cdot \nu^{d+1}}{r/a}(1+\nu r/a)^{p+\theta-d+1}\aint_{\cQ_{\nu r}(a)\cap
\Omega}|u_t+\bar{a}u_{xx}|^q\mu(dtdx)\\
&\leq&
\frac{N}{(1+\nu r/a)^q}\aint_{\cQ_{\nu r}(a)\cap
\Omega}|u_{xx}(t,x)|^q\mu(dtdx)
+\frac{N \cdot \nu^{d+1}}{r/a}(1+\nu r/a)^{p+\theta-d+1}\left(\aint_{\cQ_{\nu r}(a)\cap
\Omega}|f|^q\mu(dtdx) +J\right),
\end{eqnarray*}
 where\begin{eqnarray*}
J &:=& \aint_{\cQ_{\nu r}(a)\cap \Omega}|(a^{ij}-\bar{a}^{ij})u_{x^ix^j}|^q \,\mu(dtdx)\leq N(\nu r)^{-\alpha-d-2}\int_{\cQ_{\nu r}(a)\cap \Omega}|(a^{ij}-\bar{a}^{ij})u_{x^ix^j}|^q \,\mu(dtdx) \\
&=&N(\nu r)^{-\alpha-d-2} \int_{\cQ_{\nu r}(a) \cap \cQ_{r_0}(t_0,a_0,x'_0)}|(a^{ij}-\bar{a}^{ij})u_{x^ix^j}|^q\mu(dtdx)\leq N(\nu r)^{-\alpha-d-2}J^{1/\beta}_1 J^{1/\beta'}_2,
\end{eqnarray*}
$$
J_1:=\int_{\cQ_{\nu r}(a)\cap \Omega}|u_{xx}|^{q\beta}\mu(dtdx)\leq N(\nu r)^{\alpha+d+2} \aint_{\cQ_{\nu r}(a)\cap \Omega}|u_{xx}|^{q\beta}\mu(dtdx) \leq  N(\nu r)^{\alpha+d+2} \bM(|u_{xx}|^{\beta q})(t,x),
$$
$$
J_2:=\int_{\cQ_{\nu r}(a)\cap \cQ_{r_0}(t_0,x_0,x'_0)}|a^{ij}-\bar{a}^{ij}|^{q\beta'}\mu(dtdx)\leq N\int_{\cQ_{\nu r}(a)\cap \cQ_{r_0}(t_0,x_0,x'_0)}|a^{ij}-\bar{a}^{ij}|\mu(dtdx) .
$$
If $|\cQ_{r_0}(t_0,x_0,x'_0)|< 2^{d+2\alpha+3}|\cQ_{\nu r}(a)\cap \Omega|$, then
\begin{eqnarray*}
\int_{\cQ_{\nu r}(a)\cap \cQ_{r_0}(t_0,x_0,x'_0)}|a^{ij}-\bar{a}^{ij}|\mu(dtdx)&\leq&\int_{ \cQ_{r_0}(t_0,x_0,x'_0)}|a^{ij}-(a^{ij})_{\cQ_{r_0}(t_0,x_0,x'_0)}|\mu(dtdx)\\
&=&\mu(\cQ_0)\aint_{ \cQ_{r_0}(t_0,x_0,x'_0)}|a^{ij}-(a^{ij})_{\cQ_{r_0}(t_0,x_0,x'_0)}|\mu(dtdx)\\
&\leq& N (\nu r)^{\alpha+d+2}a^{\#(\theta)}_{R,\kappa},
\end{eqnarray*}
and if $|\cQ_{r_0}(t_0,x_0,x'_0)| \geq 2^{d+2\alpha+3}|\cQ_{\nu r}(a)\cap \Omega|$, then
\begin{eqnarray*}
\int_{\cQ_{\nu r}(a)\cap \cQ_{r_0}(t_0,x_0,x'_0)}|a^{ij}-\bar{a}^{ij}|\mu(dtdx)
&\le&\int_{ \cQ_{\nu r}(\bar{a})}|a^{ij}-(a^{ij})_{\cQ_{\nu r}(\bar{a})}|\mu(dtdx)\\
&=&\mu(\cQ_{\nu r}(\bar{a}))\aint_{ \cQ_{\nu r}(\bar{a})}|a^{ij}-(a^{ij})_{\cQ_{\nu r}(\bar{a})}|\mu(dtdx)\\
&\leq& N (\nu r)^{\alpha+d+2}a^{\#(\theta)}_{2R,\kappa}.
\end{eqnarray*}
It follows  that
$$
J\leq N (a^{\#(\theta)}_{2R,\kappa})^{1/\beta'} \cdot \bM^{1/\beta}(|u_{xx}|^{\beta q})(t,x).
$$

\noindent
Remember that $r/a\geq \delta_0=\delta_0(\varepsilon)$. Thus for (\ref{5.4.01})  it is enough to take $\nu$ so that $N(1+\nu \delta_0)^{-q}\leq \varepsilon$ and observe that
$$
\frac{N \cdot \nu^{d+1}}{r/a}(1+\nu r/a)^{p+\theta-d+1}\leq N(\alpha)<\infty.
$$
The lemma is proved.
\end{proof}

\begin{corollary}
                                 \label{corollary 10.09.1}
Suppose the  the assumptions in Lemma \ref{lemma 5.4.01} are satisfied.

(i) The for any $\varepsilon>0$ and $(t,x)\in \Omega$,
$$
(u_{xx})^{\#}(t,x)\leq \varepsilon \bM^{1/q}(|u_{xx}|^q)+N\bM^{1/q}(|f|^q)(t,x)+N (a^{\#(\theta)}_{2R,\kappa})^{1/(q\beta')}\cdot\bM^{1/(q\beta)}(|u_{xx}|^{\beta q})(t,x),
$$
where $N=N(\varepsilon,\theta,q,d,\delta,K)$ is independent of $\kappa,t,x$.

(ii)
$$
\|Mu_{xx}\|^p_{\bL_{p,\theta}(-\infty,\infty)}\leq N(d,p,\delta,K)\|Mf\|^p_{\bL_{p,\theta}(-\infty,\infty)}+N(p)\, a^{\#(\theta)}_{2R,\kappa}\cdot\|Mu_{xx}\|^p_{\bL_{p,\theta}(-\infty,\infty)}.
$$

\end{corollary}

\begin{proof}
(i) is an easy consequence of Lemma \ref{lemma 5.4.01} and Jensen's inequality. To prove (ii),
  take $q$ and $\beta>1$ so that $q<p$, $q\beta'=p$, and apply Theorems \ref{FS} and \ref{HL}.
\end{proof}


The following result is a parabolic version of Lemma 3.3 of \cite{KK2}. Define $\bQ(\kappa):=\cup_{R>0}\bQ(R,\kappa)$.

\begin{lemma}
                          \label{lemma 10.09.1}
For any $\varepsilon>0$, there exist a constant $\kappa=\kappa(\varepsilon)\in (1/2,1)$ and nonnegative functions $\eta_k\in C^{\infty}_0(\bR^{d+1}_+), k=1,2,\cdots$ so that (i) on $\bR^{d+1}_+$
\begin{equation}
                     \label{eqn 10.09.2}
\sum_k \eta^p_k \geq 1, \quad \sum_k \eta_k\leq N(d), \quad \sum_k(M|\eta_{kx}|+M^2|\eta_{kxx}|+M^2|\eta_{kt}|)\leq \varepsilon;
\end{equation}
(ii)  for each $k$, $\text{supp}\,\, \eta_k \subset Q_k$ for some $Q_k\in \bQ(\kappa)$.
\end{lemma}

\begin{proof}
We modify the proof of Lemma 3.3 of \cite{KK2}. Let
$$
\bR^{d-1}=\bigcup_{k=1}^{\infty}Q'_{k},  \quad \quad \bR=\bigcup_{\ell=1}^{\infty} I_{\ell}
$$
be a decomposition of $\bR^{d-1}$ and $\bR$ into disjoint
unit cubes $Q'_{k}$ and $I_{\ell}$ respectively. Mollify the indicator
function of each $Q'_{k}$ and $I_{\ell}$ in such a way that
thus obtained functions $\chi_{k}$ and $\hat{\chi}_{\ell}$ vanish
outside of the twice dilated $Q'_{k}$ and $I_{\ell}$ respectively
(naturally,
with center of dilation being that of $Q'_{k}$ and  $I_{\ell}$ respectively).
Then (by multiplying by a large constant $c>0$ to $\chi_k$ and $\hat{\chi}_{\ell}$ if necessary)
$$
1  \leq\sum_{k}\chi^{p}_{k}\leq
\big(\sum_{k}\chi_{k}\big)^{p}\leq N_0, \quad 1 \leq\sum_{\ell}\hat{\chi}^{p}_{\ell}\leq
\big(\sum_{\ell}\hat{\chi}_{\ell}\big)^{p}\leq N_0
$$
on $\bR^{d-1}$ and $\bR$, respectively. Here  the constant  $N_0\in(0,\infty)$
depends only on $d$ and $p$.
Furthermore, by Lemma 3.2 of \cite{KK},
there exists a
nonnegative function $\xi\in C^{\infty}_{0}(\bR_{+})$
such that assertion (i) of the present lemma
holds on $\bR_+$ with the collection $\{ \xi(e^{n}x):
n\in\bZ\}$
in place of $\{\eta_{k}(x):k=1,2,...\}$.

We write $x=(x^{1},x')$, fix a   constant $r\in(0,1)$
to be specified later, and introduce
$$
\tau_{k}(x')=\chi_{k}(r x'),\quad \hat{\tau}_{\ell}(t)=\hat{\chi}_{\ell}(rt),\quad
\eta_{nk\ell}(x)=  \xi(e^{n}x^{1})\tau_{k}(e^{n}x')\hat{\tau}_{\ell}(e^{2n}t).
$$
Then
\begin{equation}
                                                \label{11.14.01}
1 \leq\sum_{n,k,\ell}\eta_{nk\ell}^{p} \leq
\big(\sum_{n,k,\ell}\eta_{nk\ell}\big)^{p}\leq N \quad \quad \text{on}\quad  \bR^{d+1}_+
\end{equation}
with  constant $N\in(0,\infty)$
depending only on $d$ and $p$.

Now, for any multi-index $\alpha=(\alpha^1,\cdots,\alpha^d)$ with $1\leq|\alpha|\leq 2$,  we have (with some constants
$c_{\beta\gamma}$)
$$
M^{|\alpha|} D^{\alpha}_x\eta_{nk\ell}(t,x)=
(x^{1})^{|\alpha|}e^{n|\alpha|}\sum_{\beta+\gamma=\alpha}
c_{\beta\gamma} \xi^{(\beta_1)}(e^{n}x^{1})
 (D^{\gamma}\tau_{k})(e^{n}x')\hat{\tau}_{\ell}(e^{2n}t),
 $$
 and
 $$
 M^2(\eta_{\ell nk})_t=(x^1)^2e^{2n}\xi(e^nx^1)\tau(e^nx')(\hat{\tau}_{\ell})'(e^{2n}t).
 $$
 Hence,
$$
\sum_{n,k,\ell}|M^{|\alpha|} D^{\alpha}_x\eta_{nk\ell}(x)|
\leq  N_0\sum_{\beta+\gamma=\alpha}
c_{\beta\gamma}I_{1}(\gamma)I_{2}(\alpha,\beta),
$$
where
$$
I_{1}(\gamma)=\sup_{x'}\sum_{k}|D^{\gamma}\tau_{k}(x')|
=r^{|\gamma|}\sup_{x'}\sum_{k}|D^{\gamma}\chi_{k}(x')|,
$$
$$
I_{2}(\alpha,\beta)=\sup_{x^1\geq0}
\sum_{n}(x^1)^{|\alpha|}e^{n|\alpha|}
|\xi^{(\beta_{1})}(e^{n}x^1)|
=\sup_{t\in\bR}
\sum_{n} e^{(n+t)|\alpha|}
|\xi^{(\beta_{1})}(e^{n+t})|.
$$
Obviously $I_{1}$ is finite. That $I_{2}$ is also finite
is seen from its representation as the supremum
of a continuous 1-periodic function. Moreover,
if   $\gamma=0$,
then $c_{\beta\gamma}\ne0$ only if $\beta_{1}=|\alpha|$,
in which case $c_{\beta\gamma}=1$ and,
 by the construction of $\xi$, we have $I_{2}(\alpha,\beta)
\leq\varepsilon $.
It follows that
\begin{equation}
                                                \label{11.15.1}
\sum_{n,k,\ell}|M^{|\alpha|} D^{\alpha}\eta_{nk\ell}(x)|
\leq  N(d)\varepsilon+N(\varepsilon,q,d)r.
\end{equation}
Similar calculus shows
\begin{equation}
                            \label{eqn 10.09.1}
\sum_{n,k,\ell}|M^{2}(\eta_{\ell nk})_t|
\leq N(\varepsilon,q,d)r.
\end{equation}
We renumber the set $\{\eta_{kn\ell}:n=0,\pm1,...,k=1,2,...,\ell=1,2,\cdots\}$
and write it as $\{\eta_{k}:k=1,2,...\}$.
Then from (\ref{11.15.1}) and  (\ref{eqn 10.09.1}) we
 see how to choose $r$ in order to satisfy
the last inequality in (\ref{eqn 10.09.2}) with $N(d)\varepsilon$
in place of $\varepsilon$.
This proves (i).

Now we prove  (ii). Let $(\alpha,\beta)\subset \bR_+$ so that $\text{supp} \xi \subset (\alpha,\beta)$. The above proofs show that
 $\text{supp}\, \eta_{0k\ell}\subset (t_{k\ell},t_{k\ell}+r_0)\times (\alpha,\beta)\times B_r(x'_{k\ell})=:Q_{0k\ell}$ for some $t_{k\ell}, x'_{k\ell},r_0, r$ with $r_0, r$ independent of $k,\ell$. By increasing $\beta$ and adjusting $r_0,r$ if necessary we may assume that $Q_{0k\ell}\in \bQ(\kappa)$ for some $\kappa\in (0,1)$, independent of $k,\ell$. Finally it is enough to note that
 $$
 \text{supp}\,\eta_{nk\ell}\subset
 (e^{-2n}t_{k\ell},e^{-2n}t_{k\ell}+e^{-2n}r_0)\times (e^{-2n}\alpha,e^{-2n}\beta)\times B_{e^{-2n}r}(e^{-2n}x'_{k\ell}):=Q_{nk\ell} \in \bQ(\kappa).
 $$
 The lemma is proved.
\end{proof}

\begin{lemma}
              \label{lemma 09.10.6}
   Let $u\in C^{\infty}(\Omega)$ and denote $f=u_t+a^{ij}u_{x^ix^j}$.

(i) There exists a constant $\kappa_0=\kappa_0(d,p,\theta,\delta,K)\in (0,1)$ so that if $\kappa\in [\kappa_0,1]$, then
\begin{equation}
                         \label{eqn 09.10.5}
\|Mu_{xx}\|^p_{\bL_{p,\theta}(-\infty,\infty)}\leq N(d,p,\delta,K,\kappa)\left(\|Mf\|^p_{\bL_{p,\theta}(-\infty, \infty)}+a^{\#(\theta)}_{\kappa}\|Mu_{xx}\|^p_{\bL_{p,\theta}(-\infty,\infty)}\right).
\end{equation}

(ii)  If $u(t,x)=0$  whenever $x^1 \geq R$, then
$$
\|Mu_{xx}\|^p_{\bL_{p,\theta}(-\infty,\infty)}\leq N(d,p,\delta,K)\left(\|Mf\|^p_{\bL_{p,\theta}(-\infty,\infty)}
+a^{\#(\theta)}_{R_{\kappa_0}, \kappa_0}\|Mu_{xx}\|^p_{\bL_{p,\theta}(-\infty,\infty)}\right),
$$
where $R_{\kappa_0}:=2R(1+\kappa_0)/(1-\kappa_0)$.

\end{lemma}

\begin{proof}
(i) Fix $\varepsilon\in (0,1)$ which will be specified later. Take $\{\eta_n:n=1,2,\cdots\}$ from Lemma \ref{lemma 10.09.1} corresponding to $\varepsilon$. Then since $\sum_n \eta^p_n\geq  1$,
\begin{eqnarray*}
&&\|Mu_{xx}\|^p_{\bL_{p,\theta}(-\infty,\infty)}\leq \sum_n \|\eta_n Mu_{xx}\|^p_{\bL_{p,\theta}(-\infty,\infty)}\\
&\leq& \sum_n \left(\|M (\eta_n u)_{xx}\|^p_{\bL_{p,\theta}(-\infty,\infty)}+\|u_x M(\eta_n)_{x}\|^p_{\bL_{p,\theta}(-\infty,\infty)}+
\|M^{-1}u M^2(\eta_n)_{xx}\|^p_{\bL_{p,\theta}(-\infty,\infty)}\right).
\end{eqnarray*}
Note that $u^n:=u\eta_n$ satisfies
$$
u^n_t+a^{ij}u^n_{x^ix^j}=f_n:=u(\eta_n)_t+2a^{ij}u_{x^i}(\eta_n)_{x^j}+a^{ij}u(\eta_n)_{x^ix^j}+f\eta_n,
$$
and by  Lemma \ref{lemma 10.09.1} we have $\text{supp}\, u^n \subset Q_n \in \bQ(\kappa)$ for some $\kappa=\kappa(\varepsilon)\in (0,1)$. Then by Corollary \ref{corollary 10.09.1},
$$
\|Mu^n_{xx}\|^p_{\bL_{p,\theta}(-\infty,\infty)}\leq N\|Mf_n\|^p_{\bL_{p,\theta}(-\infty,\infty)}+N(p,q)a^{\#(\theta)}_{\kappa}\cdot\|Mu^n_{xx}\|^p_{\bL_{p,\theta}(\infty)}.
$$
 It follows that
\begin{eqnarray*}
\|Mu_{xx}\|^p_{\bL_{p,\theta}(-\infty,\infty)}&\leq& N\varepsilon^p(\|M^{-1}u\|^p_{\bL_{p,\theta}(-\infty,\infty)}+\|u_x\|^p_{\bL_{p,\theta}(-\infty,\infty)})\\
&\;&+\;N a^{\#(\theta)}_{\kappa}\|Mu_{xx}\|^p_{\bL_{p,\theta}(-\infty,\infty)}+\varepsilon^p a^{\#(\theta)}_{\kappa} \|M^{-1}u\|^p_{\bL_{p,\theta}(-\infty,\infty)}\\
&\;&+\;\varepsilon^p a^{\#(\theta)}_{\kappa} \|u_x\|^p_{\bL_{p,\theta}(-\infty,\infty)})+N\|Mf\|^p_{\bL_{p,\theta}(-\infty,\infty)}.
\end{eqnarray*}
Since  $\|M^{-1}u\|_{L_{p,\theta}}+\|u_x\|_{L_{p,\theta}}\leq N\|Mu_{xx}\|_{L_{p,\theta}}$, we get (i)  if  $\varepsilon$ is sufficiently small.

 (ii) Now let $\text{supp}\, \eta_n\subset Q_n=\cQ_{\kappa_0}(t_0,x^1_0,x'_0)$.  Note that $u\eta_n=0$ if $Q_n \not\in Q_{\frac{1+\kappa_0}{1-\kappa_0}R,\kappa_0}$. Thus in the proof of (i), we only need to consider the case $Q_n \in \bQ_{\frac{1+\kappa_0}{1-\kappa_0}R,\kappa_0}$. Therefore  (ii) follows from Corollary \ref{corollary 10.09.1}(ii) and  the proof of (i).
\end{proof}

\mysection{$L_p$-theory  on $\bR^d_+$}
                   \label{section half spaces}

\begin{definition}\label{md}
Let $-\infty\leq S<T\leq \infty$. We write $u\in \frH^{\gamma+2}_{p,\theta}(S,T)$ if $u\in
M\bH^{\gamma+2}_{p,\theta}(S,T)$, $u(S,\cdot)\in U^{\gamma+2}_{p,\theta}$ ($u(-\infty,\cdot):=0$ if $S=-\infty$),
and for some $\tilde f\in M^{-1}\bH^{\gamma}_{p,\theta}(S, T)$ it holds
that for any $\phi\in C^{\infty}_0(\mathbb{R}^d)$
\begin{equation}\label{e}
(u(t,\cdot),\phi)= (u(S,\cdot),\phi)+ \int^t_S( \tilde
f(s,\cdot),\phi)ds,\quad t\in(S,T).
\end{equation}
In this case we write $u_t=\tilde{f}$.
The norm in $\frH^{\gamma+2}_{p,\theta}(S,T)$ is defined by
$$
\|u\|_{\frH^{\gamma+2}_{p,\theta}(S,T)}=\|M^{-1}u\|_{\bH^{\gamma+2}_{p,\theta}(S,T)}+\|Mu_t\|_{\bH^{\gamma}_{p,\theta}(S,T)} +\|u(S,\cdot)\|_{U^{\gamma+2}_{p,\theta}}.
$$
Define $\frH^{\gamma+2}_{p,\theta}(T):=\frH^{\gamma+2}_{p,\theta}(0,T)$, $\frH^{\gamma+2}_{p,\theta}:=\frH^{\gamma+2}_{p,\theta}(0,\infty)$
and $\frH^{\gamma+2}_{p,\theta,0}(T):=\frH^{\gamma+2}_{p,\theta}(T)\cap \{u:u(0)=0\}$.
\end{definition}

\begin{theorem}
                      \label{banach}
(i) The space $\frH^{\gamma+2}_{p,\theta}(S,T)$ is a Banach space.

(ii) If $T<\infty$, then   for any $u\in \frH^{\gamma+2}_{p,\theta,0}(T)$,
$$
\sup_{t\leq T}\|u(t)\|^p_{H^{\gamma+1}_{p,\theta}}\leq N(d,p,\theta,T)\|u\|^p_{\frH^{\gamma+2}_{p,\theta}(T)}.
$$
In particular, for any $t\leq T$,
\begin{equation}
              \label{eqn 12.13.9}
\|u\|^p_{\bH^{\gamma+1}_p(T)}\leq \int^T_0 \sup_{r\leq s}\|u(r)\|^p_{H^{\gamma+1}_{p,\theta}}ds\leq N\int^t_0\|u\|^p_{\frH^{\gamma+2}_{p,\theta}(s)}ds.
\end{equation}
(iii) For any nonnegative integer $n \geq \gamma+2$, the set
$$
\frH^{\gamma+2}_{p,\theta}(T) \bigcap \bigcup_{k=1}^{\infty} C([0,T],C^n_0(G_k))
$$
where $G_k=(1/k,k)\times \{|x'|<k\}$ is dense in $\frH^{\gamma+2}_{p,\theta}(T)$.
\end{theorem}
\begin{proof}
See Theorem 2.9 and Theorem 2.11 of \cite{KL2}. Actually in \cite{KL2}, (i) is proved only for $p\geq 2$ based on Theorems 4.2 and 7.2 in \cite{Kr99}. But by inspecting the proofs of Theorems 4.2 and 7.2 in \cite{Kr99} one can easily check that in our (deterministic) case the result holds for all $p>1$.
\end{proof}

\begin{remark}
                \label{remark 12.13.1}
It is easy to check that any function  $u\in \frH^{2}_{p,\theta}(-\infty,\infty)$ can be approximated by functions in $C^{\infty}_0(\Omega)$. Thus  Lemma \ref{lemma 09.10.6} holds for any $u\in \frH^{2}_{p,\theta}(-\infty,\infty)$.
\end{remark}

Here are  some interior H\"older estimates of functions in the space $\frH^{\gamma+2}_{p,\theta}(T)$.
\begin{theorem}
                      \label{thm interior}
  Let $p>2$ and assume
  $$
  2/p<\alpha<\beta\leq 1, \quad \gamma+2-\beta-d/p=k+\varepsilon,
  $$
  where $k\in \{0,1,2,\cdots\}$ and $\varepsilon \in (0,1]$. Denote $\delta=\beta-1+\theta/p$. Then for any $u\in \frH^{\gamma+2}_{p,\theta}(T)$ and multi-indices $i,j$ such that
  $|i|\leq j$ and $|j|=k$,

  (i) the functions $D^iu(t,x)$ are continuous in $[0,T]\times \bR^d_+$ and
  $$
  M^{\delta+|i|}D^iu(t,\cdot)- M^{\delta+|i|}D^iu(0,\cdot)   \in C^{\alpha/2-1/p}([0,T], C(\bR^d_+));
  $$

  (ii) there exists a constant $N=N(p,d,\alpha,\beta)$ so that
  \begin{equation}
                 \label{eqn 4.24.1}
  \sup_{t,s\leq T}    \left(\frac{\big|M^{\delta+|i|}D^i(u(t)-u(s))\big|_{C(\bR^d_+)}}{|t-s|^{\alpha/2-1/p}}  +
  \frac{\big[M^{\delta+|j|+\varepsilon}D^j(u(t)-u(s))\big]_{C^{\varepsilon}}}{|t-s|^{\alpha/2-1/p}}   \right)
  \leq NT^{(\beta-\alpha)/2} \|u\|_{\frH^{\gamma+2}_{p,\theta}(T)}.
  \end{equation}
  \end{theorem}
  \begin{proof}
  See Theorem 4.7 of \cite{Kr01}.
  \end{proof}

  Throughout this section we assume the following.

\begin{assumption}
                     \label{main assumption2}
 There exist  constants
$\delta,K>0$ so that
\begin{equation}
                    \label{assumption 111}
\delta|\xi|^2\leq a^{ij}(t,x) \xi^i\xi^j \leq K|\xi|^2, \quad \forall \xi \in \bR^d.
\end{equation}

\end{assumption}

\begin{theorem}
                                       \label{main theorem}
Let $p\in (1,\infty)$,  $\theta\in (d-1, d-1+p)$ and $T\in (0, \infty]$. Take $\kappa_0\in (0,1)$ from Lemma \ref{lemma 09.10.6}. Assume
that there exists a constant $\beta>0$ so that
\begin{equation}
                  \label{eqn 3.18.1}
|x^1b^i|+|(x^1)^2c|\leq \beta, \quad \forall t,x.
\end{equation}
(i) Then
there exists constants $\varepsilon_0, \beta_0>0$ depending only on  $d,p,\theta,\delta$ and $K$ so that if $a^{\#(\theta)}_{\kappa_0}<\varepsilon_0$ and $\beta\leq \beta_0$ then for any
  $f\in M^{-1}\bL_{p,\theta}(T)$ and $u_0\in
U^{2}_{p,\theta}$  the equation
\begin{equation}
                      \label{eqn 12.09.1}
u_t=a^{ij}u_{x^ix^j}+b^iu_{x^i}+cu+f, \quad u(0)=u_0
\end{equation}
 admits a
unique solution $u\in \frH^{2}_{p,\theta}(T)$, and for this solution we have
\begin{equation}
                        \label{a priori 12.12}
\|u\|_{\frH^{2}_{p,\theta}(T)}\leq
N\left(\|Mf\|_{\bL_{p,\theta}(T)}+\|u_0\|_{U^{2}_{p,\theta}}\right),
\end{equation}
where $N=N(p,\theta,\delta_0,K)$.

(ii)  Let $u\in \frH^{2}_{p,\theta}(T)$ be a solution of equation (\ref{eqn 12.09.1}) and $u(t,x)=0$  whenever $x^1 \geq R$. Then the  estimate (\ref{a priori 12.12}) holds true if $a^{\#(\theta)}_{R_{\kappa_0}, \kappa_0}<\varepsilon_0$, where $R_{\kappa_0}:=2R(1+\kappa_0)/(1-\kappa_0)$.
\end{theorem}

\begin{remark}
It is known (see Remark 3.6 of \cite{KL2}) that if $\theta\not \in (d-1,d-1+p)$, then  Theorem \ref{main theorem} is not true even for the heat equation $u_t=\Delta u+f$.
\end{remark}

{\bf{Proof of Theorem \ref{main theorem}}}.
As usual, we assume $u_0=0$ (see the proof of Theorem 5.1 in \cite{Kr99}). Take $N=N(d,p,\theta,\delta,K,\kappa_0)$ from (\ref{eqn 09.10.5}) and assume that $a^{\#(\theta)}_{\kappa_0}< \varepsilon_0:=1/{(2N)}$.

{\bf{Case 1}}. Let $T=\infty$ and $b^i=c=0$.   Due to Lemma \ref{lem constant} and the method of continuity, we only prove that estimate (\ref{a priori 12.12}) holds given that a solution $u\in \frH^{2}_{p,\theta}(T)$ already exists.

Define $v(t,x)=u(t,x)I_{t>0}$ and $\bar{f}=fI_{t>0}$, then  $v\in M^{-1}\bH^{2}_{p,\theta}(-\infty,\infty)$ and $v$ satisfies (see (\ref{e}))
$$
v_t=a^{ij}u_{x^ix^j}+\bar{f}, \quad (t,x)\in \bR^{d+1}_+.
$$
By Lemma \ref{lemma 09.10.6}  and Remark \ref{remark 12.13.1},
$$
\|Mv_{xx}\|_{\bL_{p,\theta}(-\infty,\infty)}\leq N\|Mf\|_{\bL_{p,\theta}(\infty)}.
$$
This certainly proves (\ref{a priori}).

{\bf{Case 2}}. \quad Let $T<\infty$ and $b^i=c=0$. The existence of solutions in $\frH^{\gamma+2}_{p,\theta}(T)$ is an  easy consequence of Case 1. Now suppose that $u\in \frH^{\gamma+2}_{p,\theta}(T)$ is a solution of (\ref{eqn 12.09.1}).
By  the result of Case 1, the equation
\begin{equation}
                       \label{eqn 3.07.1}
v_t=\Delta v +(a^{ij}u_{x^ix^j}+f-\Delta u)I_{t\leq T},  \quad t>0\,; \quad v(0,\cdot)=0
\end{equation}
has a unique solution $v\in \frH^{\gamma+2}_{p,\theta}(0,\infty)$. Then $v-u$ satisfies
$$
(v-u)_t=\Delta(v-u), \quad t\in (0,T)\,; \quad (v-u)(0,\cdot)=0.
$$
If follows from  Lemma \ref{lem constant} that $u=v$ for $t\in [0,T]$.  For $t\geq 0$, define
$$
 a^{ij}_{T}=a^{ij}I_{t\leq T}+ \delta^{ij} I_{t>T}.
$$
Then (\ref{eqn 3.07.1}) and the fact $u=v$ for $t\in [0,T]$ show that $v$ satisfies (replace $u$ by $v$ for $t\leq T$ in (\ref{eqn 3.07.1}))
\begin{equation}
                           \label{eqn 4.25.1}
v_t=a^{ij}_Tv_{x^ix^j}+fI_{t<T}, \quad t>0\,;\,\, v(0,\cdot)=0.
\end{equation}
By Case 1, $v\in \frH^{\gamma+2}_{p,\theta}(\infty)$ is the unique solution of (\ref{eqn 4.25.1}), and $u=v$ on $[0,T]$ whenever $u$ is a solution of  (\ref{eqn 12.09.1}) on $[0,T]$.
This obviously yields the uniqueness.

{\bf{Case 3}}. General case.   Again we only prove that there exists $\beta_0$ so that if $a^{\#(\theta)}_{\kappa_0}<\varepsilon_0$ and $\beta\leq \beta_0$ then estimate (\ref{a priori 12.12}) holds given that a solution $u\in \frH^{2}_{p,\theta}(T)$ already exists.  Obviously by the results of Case 1 and 2,
\begin{eqnarray*}
\|M^{-1}u\|_{\bH^2_{p,\theta}(T)} &\leq& N\|M(b^iu_{x^i}+cu+f)\|_{\bL_{p,\theta}(T)}\\
&\leq& N \sup |x^1 b^i|\|u_x\|_{\bL_{p,\theta}(T)}+N\sup |(x^1)^2 M^{-1}u\|_{\bL_{p,\theta}(T)}+N \|Mf\|_{\bL_{p,\theta}(T)}\\
&\leq&N \beta \|M^{-1}u\|_{\bH^2_{p,\theta}(T)} + N \|Mf\|_{\bL_{p,\theta}(T)}.
\end{eqnarray*}
Thus it is enough to take $\beta_0$ so that $N\beta_0\leq 1/2$. The theorem is proved. \hspace{3cm}$\Box$

%
%

\mysection{$L_p$-theory  on bounded $C^1$ domains}
                                    \label{section domains}

\begin{assumption}
                                         \label{assumption domain}

The domain $\cO$  is of class $C^{1}_{u}$. In other words, there exist constants $r_0, K_0\in(0,\infty)$ so that for any
$x_0 \in \partial \cO$ there exists  a one-to-one continuously differentiable mapping $\Psi$ of
 $B_{r_0}(x_0)$ onto a domain $J\subset\bR^d$ such that

(i) $J_+:=\Psi(B_{r_0}(x_0) \cap \cO) \subset \bR^d_+$ and
$\Psi(x_0)=0$;

(ii)  $\Psi(B_{r_0}(x_0) \cap \partial \cO)= J \cap \{y\in
\bR^d:y^1=0 \}$;

(iii) $\|\Psi\|_{C^{1}(B_{r_0}(x_0))}  \leq K_0 $ and
$|\Psi^{-1}(y_1)-\Psi^{-1}(y_2)| \leq K_0 |y_1 -y_2|$ for any $y_i
\in J$;

(iv)   $\Psi_{x}$ is uniformly continuous in for $B_{r_{0}}(x_{0})$.
\end{assumption}

To proceed further we introduce some well known results from
\cite{GH} and \cite{KK2} (see also \cite{La} for the details).
Denote $\rho(x):=\text{dist}(x,\partial \cO)$.

\begin{lemma}
                                           \label{lemma 10.3.1}
Let the domain $\cO$ be of class $C^{1}_{u}$. Then

(i) there is a bounded real-valued function $\psi$ defined in
$\bar{\cO} $  such that the functions $\psi(x)$ and
$\rho(x)$ are comparable. In other words,  $N^{-1}\rho(x) \leq
\psi(x) \leq N\rho(x)$ with some constant
 $N$ independent of $x$,

 (ii) for any  multi-index $\alpha$,
\begin{equation}
                                                             \label{03.04.01}
\sup_{\cO} \psi ^{|\alpha|}(x)|D^{\alpha}\psi_{x}(x)| <\infty.
\end{equation}

\end{lemma}

First we introduce  Banach spaces $H^{\gamma}_{p,\theta}(\cO)$,
 which correspond to the spaces $H^{\gamma}_{p,\theta}$ on $\bR^d_+$.
 Take  $\zeta\in C^{\infty}_{0}(\bR_{+})$ satisfying (\ref{eqn 5.6.5}), which is
$$
\sum_{n=-\infty}^{\infty}\zeta(e^{n+x})>c>0, \quad \forall x\in \bR.
$$
 For $x\in \cO$ and $n\in\bZ=\{0,\pm1,...\}$
define
$$
\zeta_{n}(x)=\zeta(e^{n}\psi(x)).
$$
Then  we have $\sum_{n}\zeta_{n}\geq c$ in $\cO$ and
\begin{equation*}
\zeta_n \in C^{\infty}_0(\cO), \quad |D^m \zeta_n(x)|\leq
N(m)e^{mn}.
\end{equation*}
For $\theta,\gamma \in \bR$, let $H^{\gamma}_{p,\theta}(\cO)$ be the
set of all distributions $u$  on $\cO$ such
that
\begin{equation}
                                                 \label{10.10.03}
\|u\|_{H^{\gamma}_{p,\theta}(\cO)}^{p}:= \sum_{n\in\bZ} e^{n\theta}
\|\zeta_{-n}(e^{n} \cdot)u(e^{n} \cdot)\|^p_{H^{\gamma}_p} < \infty.
\end{equation}

It is known (see, for instance, \cite{Lo2} or \cite{KK2}) that up to equivalent
norms the space $H^{\gamma}_{p,\theta}(\cO)$ is independent of the
choice of $\zeta$ and $\psi$. Moreover if $\gamma=n$ is a
non-negative integer then
\begin{equation}
                              \label{eqn 02.09.1}
\|u\|^p_{H^{\gamma}_{p,\theta}(\cO)} \sim
\sum_{|\alpha|\leq n}\int_{\cO} |\rho^{|\alpha|}D^{\alpha}u(x)|^p
\rho^{\theta-d}(x) \,dx.
\end{equation}
Recall that if $\gamma=n$, then the space $H^{\gamma}_{p,\theta}$  is the collection of functions $u$ on $\bR^d_+$ so that
$$
\sum_{|\alpha|\leq n}\int_{\bR^d_+} |(x^1)^{|\alpha|}D^{\alpha}u(x)|^p
(x^1)^{\theta-d}(x) \,dx <\infty.
$$

Denote
$\psi(x,y)=\psi(x)\wedge \psi(y)$. For
  $n \in\bZ$, $\mu \in(0,1]$
 and $k=0,1,2,...$, define
$$
|u|_{C}=\sup_{\cO}|u(x)|, \quad [u]_{C^{\mu}}=\sup_{x\neq
y}\frac{|u(x)-u(y)|}{|x-y|^{\mu}}.
$$
\begin{equation}
                           \label{eqn 5.6.2}
[u]^{(n)}_{k}=[u]^{(n)}_{k,\cO} =\sup_{\substack{x\in \cO\\
|\beta|=k}}\psi^{k+n}(x)|D^{\beta}u(x)|,
\end{equation}
\begin{equation}
                              \label{eqn 5.6.3}
[u]^{(n)}_{k+\mu}=[u]^{(n)}_{k+\mu,\cO} =\sup_{\substack{x,y\in \cO
\\ |\beta|=k}}
\psi^{k+\mu+n}(x,y)\frac{|D^{\beta}u(x)-D^{\beta}u(y)|}
{|x-y|^{\mu}},
\end{equation}
$$
|u|^{(n)}_{k}=|u|^{(n)}_{k,\cO}=\sum_{j=0}^{k}[u]^{(n)}_{j,\cO},
\quad |u|^{(n)}_{k+\mu}=
 |u|^{(n)}_{k+\mu,\cO}=|u|^{(n)}_{k, \cO}+
[u]^{(n)}_{k+\mu,\cO}.
$$

Below we collect some other properties of spaces
$H^{\gamma}_{p,\theta}(\cO)$ taken from \cite{Lo2} (also see \cite{KK2}).

\begin{lemma}
\label{lemma collection domain}
Let $d-1<\theta<d-1+p$.

(i) Assume that $\gamma-d/p=m+\nu$ for some $m=0,1,\cdots$ and
$\nu\in (0,1]$.  Then for any $u\in H^{\gamma}_{p,\theta}(\cO)$ and $i\in
\{0,1,\cdots,m\}$, we have
$$
|\psi^{i+\theta/p}D^iu|_{C}+[\psi^{m+\nu+\theta/p}D^m
u]_{C^{\nu}}\leq c \|u\|_{ H^{\gamma}_{p,\theta}(\cO)}.
$$

(ii) Let $\alpha\in \bR$, then
$\psi^{\alpha}H^{\gamma}_{p,\theta+\alpha
p}(\cO)=H^{\gamma}_{p,\theta}(\cO)$,
$$
\|u\|_{H^{\gamma}_{p,\theta}(\cO)}\leq c
\|\psi^{-\alpha}u\|_{H^{\gamma}_{p,\theta+\alpha p}(\cO)}\leq
c\|u\|_{H^{\gamma}_{p,\theta}(\cO)}.
$$

(iii) There is a constant $c=c(d,p,\gamma,\theta)$ so that
$$
\|a f\|_{H^{\gamma}_{p,\theta}(\cO)}\leq
c|a|^{(0)}_{|\gamma|_+}|f|_{H^{\gamma}_{p,\theta}(\cO)}.
$$

(iv) $\psi D, D\psi: H^{\gamma}_{p,\theta}(\cO)\to
H^{\gamma-1}_{p,\theta}(\cO)$ are bounded linear operators, and
$$
\|u\|_{H^{\gamma}_{p,\theta}(\cO)}\leq
c\|u\|_{H^{\gamma-1}_{p,\theta}(\cO)}+c \|\psi
Du\|_{H^{\gamma-1}_{p,\theta}(\cO)}\leq c
\|u\|_{H^{\gamma}_{p,\theta}(\cO)},
$$
$$
\|u\|_{H^{\gamma}_{p,\theta}(\cO)}\leq
c\|u\|_{H^{\gamma-1}_{p,\theta}(\cO)}+c \|D\psi
u\|_{H^{\gamma-1}_{p,\theta}(\cO)}\leq c
\|u\|_{H^{\gamma}_{p,\theta}(\cO)}.
$$

\end{lemma}

Denote
$$
\bH^{\gamma}_{p,\theta}(\cO,T)=L_p(
[0,T],H^{\gamma}_{p,\theta}(\cO)), \quad \bL_{p,\theta}(\cO,T)=\bH^{0}_{p,\theta}(\cO,T)
$$
$$
U^{\gamma}_{p,\theta}(\cO)=
\psi^{1-2/p}H^{\gamma-2/p}_{p,\theta}(\cO)).
$$

\begin{definition}
We write   $u\in \frH^{\gamma+2}_{p,\theta}(\cO,T)$  if
 $u\in \psi\bH^{\gamma+2}_{p,\theta}(\cO,T)$,
$u(0,\cdot) \in U^{\gamma+2}_{p,\theta}(\cO)$ and  for some $f \in
\psi^{-1}\bH^{\gamma}_{p,\theta}(\cO,T)$,  it holds that  $u_t=f$
in the sense of distributions, that is for any $\phi\in C^{\infty}_0(\cO)$, the equality
$$
(u(t),\phi)=(u(0),\phi)+\int^t_0 (f(s),\phi)ds
$$
holds for all $t\leq T$.
The norm in  $
\frH^{\gamma+2}_{p,\theta}(\cO,T)$ is introduced by
$$
\|u\|_{\frH^{\gamma+2}_{p,\theta}(\cO,T)}=
\|\psi^{-1}u\|_{\bH^{\gamma+2}_{p,\theta}(\cO,T)} + \|\psi
u_t\|_{\bH^{\gamma}_{p,\theta}(\cO,T)}  +
\|u(0,\cdot)\|_{U^{\gamma+2}_{p,\theta}(\cO)}.
$$
Denote  $\frH^{\gamma+2}_{p,\theta,0}(\cO,T)= \frH^{\gamma+2}_{p,\theta}(\cO,T) \cap \{u: u(0)=0\}$.
\end{definition}

\begin{lemma}
                             \label{lemma 15.05}
 There exists a constant
$N=N(d,p,\theta,\gamma,T)$ such that for any $u\in \frH^{\gamma+2}_{p,\theta,0}(T)$,
$$
 \sup_{t\leq T}\|u(t)\|_{H^{\gamma+1}_{p,\theta}(\cO)}\leq N
\|u\|_{\frH^{\gamma+2}_{p,\theta}(\cO,T)}.
$$
In particular, for any $t\leq T$,
$$
\|u\|^p_{\bH^{\gamma+1}_{p,\theta}(\cO,t)}\leq  N \int^t_0
\|u\|^p_{\frH^{\gamma+2}_{p,\theta}(\cO,s)}ds.
$$
\end{lemma}

\begin{proof}
See inequality (2.21) of \cite{Lo3}. Actually there is a restriction $p\geq 2$ in (2.21) of \cite{Lo3}, but by inspecting the proofs of Theorems 4.2 and Theorem 7.1 in \cite{Kr99} one can easily check that in our (deterministic) case the the result holds for all $p>1$.
\end{proof}

Denote $B_r(x):=\{y\in \bR^d: |x-y|<r\}$ and
$Q_r(t,x):=(t,t+r^2)\times B_r(x)$. As before, we define  weighted
mean oscillation on $Q_r(x)$ with respect to measure
$\nu(dx)=\rho^{\theta-d+p}dx$
\begin{eqnarray*}
osc^{\theta}_x(a,Q_r(t,x))&=&r^{-2}\int^{t+r^2}_t \left(\aint_{B_r(x)}\aint_{B_r(x)} |a(s,y)-a(s,z)|\nu(dy)\nu(dz)\right)ds\\
&=&\frac{1}{r^2(\nu(B_r(x)))^2}\int^{t+r^2}_t \left(\int_{B_r(x)}\int_{B_r(x)} |a(s,y)-a(s,z)|\nu(dy)\nu(dz)\right)ds.
\end{eqnarray*}
Denote $\cO_{R}:=\{x\in \cO: \rho(x)>R\}$ and $\cO^c_R:=\cO \setminus \cO_R$. For $\kappa \in (0,1]$ and $R>0$, let $\mathbf{Q}(R,\kappa)$ be the collection of all $Q_r(t,x)$ so that  $r\leq \kappa \rho(x)$ and $Q_r(t,x)\subset \bR\times \cO^c_{R}$.
Define
$$
a^{\#(\theta)}_{R,\kappa}=a^{\#(\theta)}_{R,\kappa,\cO}=\sup_{\cQ\in \mathbf{Q}(R,\kappa)} \, osc^{\theta}_x(a,\cQ), \quad \quad a^{\#(\theta)}_{\kappa}=a^{\#(\theta)}_{\kappa,\cO}=\sup_{R>0}a^{\#(\theta)}_{R,\kappa}.
$$

Recall that
$$
osc_x(a,Q_r(t,x))=\frac{1}{r^2|B_r(x)|^2} \int^{t+r^2}_t \left(\int_{B_r(x)}\int_{B_r(x)} |a(s,y)-a(s,z)|dydz\right)ds.
$$
 For a subset $\cU\subset \cO$  we say that  $a=(a^{ij})$  is VMO in $\cU$ if
$$
\lim_{r\to 0}\sup_{Q_r(t,x)\cap \cU \neq \emptyset} osc_x(a,Q_r(t,x))=0.
$$

Throughout this section we assume the following.

\begin{assumption}
                     \label{main assumption2}
 There exist  constants
$\delta,K>0$ so that
\begin{equation}
\delta|\xi|^2\leq a^{ij}(t,x) \xi^i\xi^j \leq K|\xi|^2, \quad \forall \xi \in \bR^d.\nonumber
\end{equation}

\end{assumption}

 Here is the main result of this article.

\begin{theorem}
                    \label{thm main 2}
 Assume
\begin{eqnarray}
&& a=(a^{ij}) \,\text{is  VMO in}\,\, \cO_{\varepsilon} \quad \text{for any}\, \varepsilon>0  \label{con 1}\\
&&
\lim_{\rho(x) \to 0} \sup_t \left(\rho(x)|b^i(t,x)|+\rho^2(x)|c(t,x)|\right)=0 \label{con 2}
\end{eqnarray}
Then there exist constants $\varepsilon_1, \kappa_1\in (0,1)$ so that if
\begin{equation}
                        \label{eqn hard}
\lim_{R\to 0}a^{\#(\theta)}_{R, \kappa_1}<\varepsilon_1,
\end{equation}
then
 for any
  $f\in \psi^{-1}\bL_{p,\theta}(\cO,T)$ and $u_0\in
U^{2}_{p,\theta}(\cO)$  the equation
\begin{equation}
                      \label{eqn 12.09.1}
u_t=a^{ij}u_{x^ix^j}+b^iu_{x^i}+cu+f, \quad u(0)=u_0
\end{equation}
 admits a
unique solution $u\in \frH^{2}_{p,\theta}(\cO,T)$, and for this solution we have
\begin{equation}
                        \label{a priori}
\|u\|_{\frH^{2}_{p,\theta}(\cO,T)}\leq
N\left(\|\psi f\|_{\bL_{p,\theta}(\cO,T)}+\|u_0\|_{U^{2}_{p,\theta}(\cO)}\right),
\end{equation}
where $N=N(p,\theta,\delta_0,K,T)$.
\end{theorem}

\begin{remark}
                \label{main remark 2}
$(i)$ By inspecting our proof,  one  easily checks that (\ref{eqn hard}) and   (\ref{con 1}) can be  replaced by
$$
a^{\#(\theta)}_{R, \kappa_1}<\varepsilon_1 \quad \text{for some}\,\, R>0, \quad \quad \text{and} \quad a=(a^{ij}) \,\, \text{is  VMO in}\, \cO_R.
$$

\noindent
$(ii)$ Obviously,  (\ref{eqn hard}) and   (\ref{con 1}) are certainly satisfied if $a$ is VMO in $\cO$
 (see  Remark \ref{main remark}).

\noindent
$(iii)$ Our proof shows that  (\ref{con 2}) can be replaced by
$$
\lim_{\rho(x) \to 0} \sup_t \left(\rho(x)|b^i(t,x)|+\rho^2(x)|c(t,x)|\right)<\beta
$$
for some $\beta>0$.
\end{remark}

{\bf{Proof of Theorem \ref{thm main 2}}}

See Theorem  2.10 of \cite{KK2} for the case  $a^{ij}=\delta^{ij}, b^i=c=0$.  Hence, due to the method of continuity, we only need to show that  (\ref{a priori}) holds given that a solution $u\in \frH^{2}_{p,\theta}(T)$ already exists.  Let $u\in \frH^{2}_{p,\theta}(\cO,T)$ be a solution of equation  (\ref{eqn 12.09.1}). By Theorem 2.10 of \cite{KK2}, the equation
$$
v_t=\Delta v, \quad v(0)=u_0
$$
has a unique solution $v\in \frH^2_{p,\theta}(\cO,T)$, and furthermore
$$
\|v\|_{\frH^2_{p,\theta}(\cO,T)}\leq N\|u_0\|_{U^2_{p,\theta}(\cO)}.
$$
Thus considering $u-v$, we assume   $u_0=0$.

 Let $x_0 \in \partial \cO$ and $\Psi$ be a function from
Assumption \ref{assumption domain}.
In \cite{KK2} it is shown that
 $\Psi$ can be chosen in such a way that  for any non-negative integer $n$
\begin{equation}
                                                          \label{2.25.03}
|\Psi_{x}|^{(0)}_{n,B_{r_0}(x_0)\cap \cO} +
 |\Psi^{-1}_{x}|^{(0)}_{n,J_{+}} < N(n)<  \infty
\end{equation}
and
\begin{equation}
                                                     \label{2.25.02}
\rho(x)\Psi_{xx}(x) \to 0 \quad \text{as}\quad  x\in
B_{r_0}(x_0)\cap \cO,
 \text{and} \,\,\,  \rho(x) \to 0,
\end{equation}
where the constants $N(n)$ and the
 convergence in (\ref{2.25.02}) are independent of  $x_0$.
Define $r=r_{0}/K_{0}$
and fix smooth functions $\eta \in C^{\infty}_{0}(B_{r/2}(0) )$ such that $ 0 \leq \eta \leq 1$, and
 $\eta=1$ in $B_{r/4}(0)$. Observe that
$\Psi(B_{r_0}(x_0))$ contains $B_r $ and $a^{ij}(\Psi^{-1}(x))$ is well defined for any $x\in B_r(0)$.
 For  $t>0$, $x\in\bR^{d}_{+}$ let us introduce
$$
\hat{a}^{ij}(t,x):= \eta(x)\left(\sum_{l,m=1}^d
  a^{lm}(t,\Psi^{-1}(x))\partial_l\Psi^{i}(\Psi^{-1}(x))\partial_m\Psi^{j}(\Psi^{-1}(x))\right)  +
\delta^{ij}(1- \eta(x)),
$$
\begin{eqnarray*}
\hat{b}^{i}(t,x) &:=&\eta(x) \Big[
\sum_{l,m}a^{lm}(t,\Psi^{-1}(x))\cdot
\partial_{lm}\Psi^{i}(\Psi^{-1}(x))+\sum_{l}b^{l}(t,\Psi^{-1}(x))\cdot\partial_l\Psi^{i}(\Psi^{-1}(x))\Big],
\end{eqnarray*}
$$
 \hat{c}(t,x) :=\eta(x) c(t,\Psi^{-1}(x)).
 $$
Then  by (\ref{con 2}) and  (\ref{2.25.02}) one can easily find
$r_1>0$ satisfying
 $$
 \sup_{x^1\leq K_0 r_1}\left(|x^1\hat{b}^i|+(x^1)^2\hat{c}|\right)\leq \beta_0/2.
 $$
 Denote $$\bar{b}^i=\hat{b}^iI_{x^1\leq K_0r_1}, \quad \bar{c}=\hat{c}I_{x^1\leq K_0r_1}.
 $$

 It is not hard to check that there exists $\kappa_1\in (0,1)$ so that (if $R$ is sufficiently small)
 $$
 \hat{a}^{\#(\theta)}_{R,\kappa_0, \bR^d_+}\leq N a^{\#(\theta)}_{K_0R,\kappa_1, \cO}+ N \eta^{\#(\theta)}_{R,\kappa_0, \bR^d_+}+ c(R),
 $$
 where $N=N(K_0,\eta)$ and $c(R)\downarrow 0$ as $R\to 0$. Take $\varepsilon_1$ so that $N\varepsilon_1\leq \varepsilon_0/2$. The if
 $\lim_{R\to 0} a^{\#(\theta)}_{R,\kappa_1, \cO}\leq \varepsilon_1$
 then for any sufficiently small $R$, we have $\hat{a}^{\#(\theta)}_{R_0,\kappa_0, \bR^d_+}\leq \varepsilon_0$, where $R_0=2R(1+\kappa_0)/(1-\kappa_0)$.

 Denote $\bar{r}=r/{(4K_0)}\wedge r_1 \wedge R_0/{K_0}$. Let $\zeta$ be a smooth function
with support in $B_{\bar{r}}(x_0)$ and denote
$v:=(u\zeta)(\Psi^{-1})$ and continue $v$ as zero in
$\bR^{d}_{+}\setminus\Psi(B_{\bar{r}}(x_0))$. Since
$\eta=1$ on $\Psi(B_{\bar{r}}(x_0))$, the function  $v$
satisfies
$$
 v_t = \hat{a}^{ij}v_{x^i x^j} + \bar{b}^{i}v_{x^i} +
\bar{c}v + \hat{f}
$$
where
$$\hat{f} =\tilde{f}^k(\Psi^{-1}), \quad \tilde{f}^k=
-2a^{ij}u_{x^{i}}\zeta_{x^{j}}
-a^{ij}u\zeta_{x^{i}x^{j}}-b^{i}u^r\zeta_{x^{i}} +\zeta f.
$$

Next we observe  that by Lemma \ref{lemma 10.3.1} and
  Theorem 3.2 in \cite{Lo2} (or see \cite{KK2})
for any $\nu,\alpha \in \bR $ and $h \in
\psi^{-\alpha}H^{\nu}_{p,\theta}(\cO)$ with support in
$B_{\bar{r}}(x_0)$
\begin{equation}
                                                          \label{1.28.01}
\|\psi^{\alpha}h\|_{H^{\nu}_{p,\theta}(\cO)} \sim
\|M^{\alpha}h(\Psi^{-1})\|_{H^{\nu}_{p,\theta}}.
\end{equation}
Therefore we  conclude that
 $v\in \frH^{2}_{p,\theta}(T)$, and  by Theorem \ref{main theorem}(ii)  we have, for any $t\leq T$,
$$
\|M^{-1}v\|_{\bH^{2}_{p,\theta}(t)} \leq N \|M\hat{f}
\|_{\bL_{p,\theta}(t)}.
$$
By using (\ref{1.28.01}) again we obtain
 \begin{eqnarray*}
\|\psi^{-1}u\zeta\|_{\bH^2_{p,\theta}(\cO,t)} &\leq& N
\|a\zeta_x \psi u_x\|_{\bL_{p,\theta}(\cO,t)} + N
\|a\zeta_{xx}\psi u\|_{\bL_{p,\theta}(\cO,t)}\\
&+& N \|\zeta_x \psi b u\|_{\bL_{p,\theta}(\cO,t)}  +  N \|\zeta \psi
f\|_{\bL_{p,\theta}(\cO,t)}.
\end{eqnarray*}

Next, we easily check that
$$
\sup_{t,x} \left(|\zeta_x a|+ |\zeta_{xx}\psi
a|+ |\zeta_x \psi
b|\right)<\infty
$$
 and    conclude
$$
\|\psi^{-1}u\zeta\|_{\bH^{2}_{p,\theta}(\cO,t)} \leq N \|\psi
u_x\|_{\bL_{p,\theta}(\cO,t)} + N
\|u\|_{\bL_{p,\theta}(\cO,t)}
+N \|\psi f\|_{\bL_{p,\theta}(\cO,t)}.
$$

Finally, to estimate the norm
 $\|\psi^{-1} u\|_{\bH^{2}_{p,\theta}(\cO,t)}$,
 we introduce a partition of unity $\zeta_{(i)}, i=0,1,2,...,M$ such
that $\zeta_{(0)} \in C^{\infty}_0(\cO)$ and
  $\zeta_{(i)} \in C^{\infty}_0(B_{\bar{r}}(x_i))$,
$ x_i \in \partial \cO$ for $i\geq1$.
   Observe that since
$u\zeta_{(0)}$ has compact support in $\cO$, we get
$$
\|\psi^{-1}u\zeta_{(0)}\|_{\bH^{2}_{p,\theta}(\cO,t)}\sim
\|u\zeta_{(0)}\|_{\bH^{2}_{p}(t)}.
$$
Thus we can estimate
$\|\psi^{-1} u\zeta_{(0)}\|_{\bH^{2}_{p,\theta}(\cO,t)}$ using
Theorem 2.1 in \cite{Kr07} and the other norms as above.
By summing
up those estimates we get
$$
\|\psi^{-1} u\|_{\bH^{2}_{p,\theta}(\cO,t)} \leq N
 \|\psi u_x\|_{\bL_{p,\theta}(\cO,t)}+
N\|u\|_{\bL_{p,\theta}(\cO,t)}
+ N \|\psi f\|_{\bL_{p,\theta}(\cO,t)}.
$$
Furthermore, we know   (see Lemma \ref{lemma collection domain}) that
$$
\|\psi u_x\|_{ H^{\gamma}_{p,\theta}(\cO)} \leq N \|u\|_{
H^{\gamma+1}_{p,\theta}(\cO)}.
$$
Therefore it follows
\begin{eqnarray*}
\|u\|^p_{\frH^{2}_{p,\theta}(\cO,t)} &\leq& N
\|u\|^p_{\bH^{1}_{p,\theta}(\cO,t)} + N \|\psi
f\|^p_{\bL_{p,\theta}(\cO,t)}\\
&\leq & N \int^t_0 \|u\|^p_{\frH^{2}_{p,\theta}(\cO,s)}\,ds+N \|\psi
f\|^p_{\bL_{p,\theta}(\cO,t)},
\end{eqnarray*}
where Lemma \ref{lemma 15.05} is used for the second inequality. Now
(\ref{a priori})  follows from Gronwall's inequality.  The theorem
is proved. \hspace{11cm}$\Box$


\begin{thebibliography}{mm}
{\small

\bibitem{BC} M. Bramanti and M.C. Cerutti, {\em $W^{1,2}_p$ solvability for the Cauchy-Dirichlet problem for parabolic equations with VMO coefficients\/}, Comm. Partial Differential Equations {\bf{18}} (1993), 1735-1763.


\bibitem{By1} Sun-Sig Byun, {\em Elliptic equations with BMO coefficients in Lipschitz domains\/}, Trans. Amer. Math. Soc. {\bf{357}} (2005), no. 3, 1025-1046.

\bibitem{By2} Sun-Sig Byun, {\em Parabolic eqiatopms with BMO coefficients in Lipschitz domains\/}, J. Differential Eequations {\bf{209}} (2005), no. 2, 229-265.

 \bibitem{CFL} F. Chiarenza, M. Frasca and P. Longo, {\em Interior $W^{2,p}$ estimates  for   nondivergence elliptic equations with discontinuous coefficients\/}, Ricerche Mat. {\bf{40}} (1991), no.1, 149-168.

 \bibitem{CFL-1} F. Chiarenza, M. Frasca and P. Longo, {\em  $W^{2,p}$-solvability of the Dirichlet problem for nondivergence elliptic equations with VMO coefficients\/}, Trans. Amer. Math. Soc. {\bf{336}} (1993), no.2, 841-853.


\bibitem{DK1} Hongjie Dong and Doyoon Kim, {\em Parabolic and elliptic systems in divergence form with variably partially BMO coefficients\/},  preprint arXiv:0902.0390.

\bibitem{DK2} Honggie Dong and N.V. Krylov, {\em Second-order elliptic and parabolic equations with $B(\bR^2, VMO)$ coefficients\/}, Trans. Aner. Math. Soc. {\bf{362}} (2010), 6477-6494.



\bibitem{GH} D. Gilbarg and L. H\"ormander,
{\em Intermediate Schauder estimates\/}, Archive Rational Mech.
Anal., {\bf{74}} (1980), 297-318.


\bibitem{DK3} Doyoon Kim and N.V. Krylov, {\em Elliptic differential equations with coefficients measurable with respect to one variable and VMO with respect to the otheres\/}, SIAM J. Math. Anal. {\bf{39}} (2007), no.2, 489-506.

\bibitem{Kim11}
K. Kim,
{\em A weighted Sobolev space theory for parabolic stochastic PDEs on non-smooth domains\/},
Preprint. arXiv:1109.4727.




\bibitem{Kim03}  K. Kim,   {\em On stochastic partial
differential equations with variable coefficients in $C^1$
domains\/},
  Stochastic processes and their applications, {\bf{112}} (2004), no.2, 261-283.

\bibitem{KL11} K. Kim and K. Lee, {\em A weighted $L_p$-theory for second-order elliptic and parabolic partial differential systems on a half space\/}, preprint. arXiv:1204.2325v1.


\bibitem{KK} K. Kim and N.V. Krylov, {\em On SPDEs with variable
coefficients in one space dimension\/},  Potential Anal, {\bf{21}}
(2004), no. 3, 203-239.

\bibitem{KK2} K. Kim and N.V. Krylov,  {\em On the Sobolev space theory
of parabolic and elliptic equations in $C^{1}$ domains\/}, SIAM J.
Math. Anal. {\bf{36}} (2004), 618-642.

\bibitem{Kr01}  N.V. Krylov, {\em Some properties of traces for stochastic
and determistic parabolic weighted Sobolev spaces}, Journal of
Functional Analysis {\bf{183}} (2001), 1-41.




\bibitem{kr08} N.V. Krylov, {\em Lectures on Elliptic and Parabolic Equations in Sobolev Spaces\/},
American Mathematical Society, Prividence, RI, 2008.

\bibitem{Kr07} N.V. Krylov, {\em Parabolic and Elliptic equations with VMO coefficients\/}, Comm. Partial Differential Equations {\bf{32}} (2007), no. 1-3, 453-475.



\bibitem{Kr99}  N.V. Krylov, {\em An analytic approach to SPDEs\/},
 pp. 185-242 in
Stochastic Partial Differential Equations: Six Perspectives,
Mathematical Surveys and Monographs {\bf{64}} (1999), AMS,
Providence, RI.

\bibitem{kr99} N.V. Krylov, {\em Weighted Sobolev spaces and Laplace equations
and the heat equations in a half space\/}, Comm. in PDEs {\bf{23}}
(1999), no. 9-10, 1611-1653.


\bibitem{kr99-1} N.V. Krylov, {\em Some properties of weighted Sobolev spaces in $\bR^d_+$\/}, Ann. Scuola Norm. Sup. Pisa Cl. Sci. {\bf{28}} (1999), no. 4, 675-693.





\bibitem{KL2} N.V. Krylov and S.V. Lototsky, {\em A Sobolev space
theory of SPDEs with constant coefficients in a half space\/}, SIAM
J. on Math. Anal., {\bf{31}} (1999), no. 1, 19-33.


\bibitem{Ku} Alois Kufner, {\em  Weighted Sobolev spaces\/}, John
Wiley and Sons Inc, 1984.

\bibitem{La} S.K. Lapic, {\em On the first-initial boundary
value problem for stochastic partial differential equations\/},
Ph.D. thesis, University of Minnesota, Minneapolis, MN, 1994.




\bibitem{LM} J.L. Lions and E. Magenes, {\em Probl$\grave{e}$mes aux limites non homog$\grave{e}$nes et applications, 1\/}, Dunod, Paris, 1968.

\bibitem{Lo} S.V. Lototsky, {\em Dirichlet problem for stochastic
parabolic equations in smooth domains\/}, Stochastics and
Stochastics Reports, {\bf{68}} (1999), no. 1-2,  145-175.

\bibitem{Lo2} S.V. Lototsky,
{\em Sobolev spaces with weights in domains and boundary value
problems for degenerate elliptic equations\/}, Methods
 and Applications of
Analysis {\bf{1}} (2000), no. 1, 195-204.

\bibitem{Lo3} S.V. Lototksy,
{\em Linear stochastic parabolic equations, degenerating on the
boundary of a domain\/}, Electronic Journal of Probability,
 Vol. 6, No. 24 (2001), 1-14.



\bibitem{T} H. Triebel, {\em Theory of function spaces\/},
Birkh\"auser Verlag, Basel-Boston-Stuttgart, 1983.


}
\end{thebibliography}
\end{document}